\documentclass[reqno]{amsart}
\usepackage[colorlinks=true]{hyperref}
\usepackage{amsmath}
  \usepackage{paralist}
  \usepackage{amssymb}
 \usepackage{amsthm}

\AtBeginDocument{{\noindent\small \emph{}} \vspace{8mm}}
\begin{document}
\title[\hfil Standing waves to upper critical Choquard equation]
{Standing waves to upper critical Choquard equation with a local
perturbation: multiplicity, qualitative properties and stability}

\author[X. F. Li]
{Xinfu Li}

\address{Xinfu Li \newline
School of Science, Tianjin University of Commerce, Tianjin 300134,
Peoples's Republic of China} \email{lxylxf@tjcu.edu.cn}


\begin{abstract}
In this paper, we consider the upper critical Choquard equation with
a local perturbation
\begin{equation*}
\begin{cases}
-\Delta u=\lambda u+(I_\alpha\ast|u|^{p})|u|^{p-2}u+\mu|u|^{q-2}u,\
x\in
\mathbb{R}^{N},\\
u\in H^1(\mathbb{R}^N),\ \int_{\mathbb{R}^N}|u|^2=a,
\end{cases}
\end{equation*}
where $N\geq 3$, $\mu>0$, $a>0$, $\lambda\in \mathbb{R}$, $\alpha\in
(0,N)$, $p=\bar{p}:=\frac{N+\alpha}{N-2}$, $q\in (2,2+\frac{4}{N})$
and $I_\alpha=\frac{C}{|x|^{N-\alpha}}$ with $C>0$. When $\mu
a^{\frac{q(1-\gamma_q)}{2}}\leq
(2K)^{\frac{q\gamma_q-2\bar{p}}{2(\bar{p}-1)}}$ with
$\gamma_q=\frac{N}{2}-\frac{N}{q}$ and  $K$ being some positive
constant, we prove

(1) Existence and orbital stability of the ground states.

(2) Existence, positivity, radial symmetry, exponential decay and
orbital instability of the ``second class' solutions.

This paper generalized  and improved parts of the results obtained
in \cite{{JEANJEAN-JENDREJ},{Jeanjean-Le},{Soave JFA},{Wei-Wu 2021}}
to the Schr\"{o}dinger equation.
\end{abstract}

\maketitle \numberwithin{equation}{section}
\newtheorem{theorem}{Theorem}[section]
\newtheorem{lemma}[theorem]{Lemma}
\newtheorem{definition}[theorem]{Definition}
\newtheorem{remark}[theorem]{Remark}
\newtheorem{proposition}[theorem]{Proposition}
\newtheorem{corollary}[theorem]{Corollary}
\allowdisplaybreaks

\section{Introduction and main results}

\setcounter{section}{1} \setcounter{equation}{0}

In this paper, we study standing waves of prescribed mass to the
Choquard equation with a local perturbation
\begin{equation}\label{e1.2}
i\partial_t\psi+\Delta
\psi+(I_\alpha\ast|\psi|^{p})|\psi|^{p-2}\psi+\mu|\psi|^{q-2}\psi=0,\
(t,x)\in \mathbb{R}\times \mathbb{R}^N,
\end{equation}
where $N\geq 3$, $\psi:\mathbb{R}\times \mathbb{R}^N\to \mathbb{C}$,
$\mu>0$, $\alpha\in (0,N)$,  $I_{\alpha}$ is the Riesz potential
defined for every $x\in \mathbb{R}^N \setminus \{0\}$ by
\begin{equation}\label{e1.3}
I_{\alpha}(x):=\frac{A_\alpha(N)}{|x|^{N-\alpha}},\
A_\alpha(N):=\frac{\Gamma(\frac{N-\alpha}{2})}{\Gamma(\frac{\alpha}{2})\pi^{N/2}2^\alpha}
\end{equation}
with $\Gamma$ denoting the Gamma function (see \cite{Riesz1949AM},
P.19), $p$ and $q$ will be defined later.

The equation (\ref{e1.2}) has several physical origins. When $N =
3$, $p = 2$, $\alpha = 2$ and $\mu=0$, (\ref{e1.2}) was investigated
by Pekar in \cite{Pekar 1954} to study the quantum theory of a
polaron at rest. In \cite{Lieb 1977}, Choquard applied it as an
approximation to Hartree-Fock theory of one component plasma. It
also arises in multiple particles systems \cite{Gross 1996} and
quantum mechanics \cite{Penrose 1996}. When $p = 2$, equation
(\ref{e1.2}) reduces to the well-known Hartree equation. The
Choquard equation (\ref{e1.2}) with or without a local perturbation
has attracted much attention nowadays, see \cite{{Bonanno-Avenia
2014},{Cazenave 2003},{Chen-Guo 2007},{Feng-Yuan 2015},{Liu-Shi
2018},{Miao-Xu 2010}} for the local existence, global existence,
blow up and more in general dynamical properties.

Standing waves to (\ref{e1.2}) are solutions of the form $\psi(t, x)
=e^{-i\lambda t}u(x)$, where $\lambda\in \mathbb{R}$ and
$u:\mathbb{R}^N\to \mathbb{C}$. Then $u$ satisfies the equation
\begin{equation}\label{e1.4}
-\Delta u=\lambda u+(I_\alpha\ast |u|^{p})|u|^{p-2}u+\mu|u|^{q-2}u,\
x\in \mathbb{R}^{N}.
\end{equation}
When looking for solutions to (\ref{e1.4}) one choice is to fix
$\lambda<0$ and to search for solutions to (\ref{e1.4}) as critical
points of the action functional
\begin{equation*}
J(u):=\int_{\mathbb{R}^N}\left(\frac{1}{2}|\nabla
u|^2-\frac{\lambda}{2}|u|^2-\frac{1}{2p}(I_\alpha\ast
|u|^{p})|u|^{p}-\frac{\mu}{q}|u|^q\right)dx,
\end{equation*}
see for example \cite{{Li-Ma 2020},{Li-Ma-Zhang 2019},{Luo
2020},{Moroz-Schaftingen 2017}} and the references therein. Another
choice is to fix the $L^2$-norm of the unknown $u$, that is, to
consider the problem
\begin{equation}\label{e1.5}
\begin{cases}
-\Delta u=\lambda u+(I_\alpha\ast|u|^{p})|u|^{p-2}u+\mu|u|^{q-2}u,\
x\in
\mathbb{R}^{N},\\
u\in H^1(\mathbb{R}^N),\ \int_{\mathbb{R}^N}|u|^2=a
\end{cases}
\end{equation}
with fixed $a>0$ and unknown $\lambda\in \mathbb{R}$. In this
direction, define on $H^1(\mathbb{R}^N)$ the energy functional
\begin{equation*}
E(u):=\frac{1}{2}\int_{\mathbb{R}^N}|\nabla
u|^2-\frac{1}{2p}\int_{\mathbb{R}^N}(I_\alpha\ast|u|^{p})|u|^{p}
-\frac{\mu}{q}\int_{\mathbb{R}^N}|u|^{q}.
\end{equation*}
It is standard to check that $E\in C^1$ under some assumptions on
$p$ and $q$, and a critical point of $E$ constrained to
\begin{equation*}
S_a:=\left\{u\in
H^1(\mathbb{R}^N):\int_{\mathbb{R}^N}|u|^2=a\right\}
\end{equation*}
gives rise to a solution to (\ref{e1.5}). Such solution is usually
called a normalized solution of (\ref{e1.4}) on $S_a$, which is the
aim of this paper.

For future reference, we recall

\begin{definition}\label{def1.1}
We say that $u$ is a normalized ground state to (\ref{e1.4}) on
$S_a$ if
\begin{equation*}
E(u)=c_a^{g}:=\inf\{E(v):v\in S_a,\ (E|_{S_a})'(v)=0\}.
\end{equation*}
The set of the normalized ground states will be denoted by
$\mathcal{G}_a$.
\end{definition}

\begin{definition}\label{def1.2}
$\mathcal{G}_a$ is orbitally stable if for every $\epsilon>0$ there
exists $\delta>0$ such that, for any $\psi_0\in H^1(\mathbb{R}^N)$
with $\inf_{v\in \mathcal{G}_a}\|\psi_0-v\|_{H^1}<\delta$, we have
\begin{equation*}
\inf_{v\in \mathcal{G}_a}\|\psi(t,\cdot)-v\|_{H^1}<\epsilon\ \
\mathrm{for\ any\ }t>0,
\end{equation*}
where $\psi(t,x)$ denotes the solution to (\ref{e1.2}) with initial
value $\psi_0$.

A standing wave $e^{-i\lambda t}u$ is strongly unstable if for every
$\epsilon>0$ there exists $\psi_0\in H^1(\mathbb{R}^N)$ such that
$\|\psi_0-u\|_{H^1}<\epsilon$, and $\psi(t,x)$ blows up in finite
time.
\end{definition}

When studying normalized solutions to the Choquard equation
\begin{equation}\label{e1.12}
-\Delta u=\lambda u+(I_\alpha\ast|u|^p)|u|^{p-2}u,\ x\in
\mathbb{R}^{N},
\end{equation}
the $L^2$-critical exponent $p^*:=1+\frac{2+\alpha}{N}$, the
Hardy-Littlewood-Sobolev upper critical exponent
$\bar{p}:=\frac{N+\alpha}{N-2}$ and lower critical exponent
$\underline{p}:=\frac{N+\alpha}{N}$ play an important role. For
$\underline{p}<p<p^*$, the existence of normalized ground state to
(\ref{e1.12}) was studied by Cazenave and Lions \cite{Cazenave-Lions
1982} and Ye \cite{Ye 2016} by considering the minimizer of $E$
constrained on $S_a$. \cite{Cazenave-Lions 1982} also studied the
orbital stability of the normalized ground states set by using the
concentration compactness principle. For $p^*<p<\bar{p}$, the
functional $E$ is no longer bounded from below on $S_a$.  By
considering the minimizer of $E$ constrained on the Poho\v{z}aev
set, Luo \cite{Luo 2019} obtained the existence and instability of
normalized ground state to (\ref{e1.12}). For $p=p^*$, by scaling
invariance, the result is delicate, see \cite{Cazenave-Lions 1982}
and \cite{Ye 2016} for details. See
\cite{{Bartsch-Liu-Liu_2020},{Li-Ye_JMP_2014},{Yuan-Chen-Tang_2020}}
for studies to Choquard equation with general nonlinearity. For
(\ref{e1.12}) with $p=\bar{p}$, Moroz and Van Schaftingen
\cite{Moroz-Schaftingen JFA 2013} showed that (\ref{e1.12}) has no
solutions in $H^1(\mathbb{R}^N)$ for fixed $\lambda<0$. While Gao
and Yang \cite{Gao-Yang-1} obtained the solution to the equation
\begin{equation}\label{e1.13}
-\Delta u=(I_\alpha\ast|u|^{\bar{p}})|u|^{\bar{p}-2}u,\ x\in
\mathbb{R}^N
\end{equation}
in $D^{1,2}(\mathbb{R}^N)$. So it is interesting to study the
 normalized solutions to (\ref{e1.12}) with $p=\bar{p}$
under a local perturbation $\mu|u|^{q-2}u$, namely equation
(\ref{e1.5}). In a recent paper, Li \cite{Li-2021} considered the
existence and symmetry of solutions to (\ref{e1.5}) with $p=\bar{p}$
and $2+\frac{4}{N}\leq q<2^*:=\frac{2N}{N-2}$. Note that
$2+\frac{4}{N}$ is the $L^2$-critical exponent in studying
normalized solutions to the Schr\"{o}dinger equation
\begin{equation*}
-\Delta u=\lambda u+|u|^{q-2}u,\ x\in \mathbb{R}^{N}.
\end{equation*}
In this paper, we consider (\ref{e1.5}) with $p=\bar{p}$ and
$2<q<2+\frac{4}{N}$. This paper is motivated by
\cite{{JEANJEAN-JENDREJ},{Jeanjean-Le},{Soave JFA},{Wei-Wu 2021}}
which considered normalized solutions to the Schr\"{o}dinger
equation with mixed nonlinearities
\begin{equation}\label{e7.28}
-\Delta u=\lambda u+|u|^{2^*-2}u+\mu|u|^{q-2}u,\ x\in
\mathbb{R}^{N}.
\end{equation}
We should point out that Liu and Shi \cite{Liu-Shi 2018} studied the
existence and orbital stability of ground states to (\ref{e1.5})
with $p=p^*$ and $2<q<2+\frac{4}{N}$.

Before stating the main results of this paper, we make some
notations. In the following, we assume
$p=\bar{p}:=\frac{N+\alpha}{N-2}$ in (\ref{e1.5}). Set
\begin{equation}\label{e3.14}
\begin{split}
S_\alpha:&=\inf_{ u\in
D^{1,2}(\mathbb{R}^N)\setminus\{0\}}\frac{\int_{\mathbb{R}^N}|\nabla
u|^2}{\left(\int_{\mathbb{R}^N}\left(I_{\alpha}\ast
|u|^{\bar{p}}\right)|u|^{\bar{p}}\right)^{{1}/{\bar{p}}}},
\end{split}
\end{equation}
\begin{equation}\label{e1.14}
K:=
\frac{2\bar{p}-q\gamma_q}{2\bar{p}(2-q\gamma_q)S_\alpha^{\bar{p}}}
\left(\frac{\bar{p}(2-q\gamma_q)C_{N,q}^qS_\alpha^{\bar{p}}}{q(\bar{p}-1)}\right)^{\frac{2\bar{p}-2}{2\bar{p}-q\gamma_q}}
\end{equation}
with $\gamma_q:=\frac{N}{2}-\frac{N}{q}$ and $C_{N,q}$ defined in
Lemma \ref{lem2.5},
\begin{equation}\label{e**}
\rho_0:=\left(\frac{\bar{p}(2-q\gamma_q)S_\alpha^{\bar{p}}}{2\bar{p}-q\gamma_q}\right)^{\frac{1}{\bar{p}-1}},
\end{equation}
\begin{equation*}
B_{\rho_0}:=\{u\in H^1(\mathbb{R}^N):\|\nabla u\|_2^2<\rho_0\},\ \ \
V_a:=S_a\cap B_{\rho_0},\ \ m_a:=\inf_{u\in V_a}E(u).
\end{equation*}

Now we state the first two main results of this paper.

\begin{theorem}\label{thm6.1}
Let $N\geq 3$, $\alpha\in (0,N)$, $2<q<2+\frac{4}{N}$, $p=\bar{p}$,
$\mu>0$, $a>0$, $\mu a^{\frac{q(1-\gamma_q)}{2}}\leq
(2K)^{\frac{q\gamma_q-2\bar{p}}{2(\bar{p}-1)}}$. Then

(1) $E|_{S_a}$ has a critical point $\tilde{u}$ at negative level
$m_a<0$ which is an interior local minimizer of $E$ on the set
$V_a$.

(2) $m_a=c_a^g$ (that is, $\tilde{u}$ is a ground state to
(\ref{e1.5})), and any other ground state to (\ref{e1.5}) is a local
minimizer of $E$ on $V_a$.

(3) $\mathcal{G}_a$ is compact, up to translation.

(4) $c_a^g$ is reached by a positive and  radially symmetric
non-increasing function.

(5) For any $u\in \mathcal{G}_a$, there exists $\lambda<0$ such that
$u$ satisfies (\ref{e1.5}).
\end{theorem}

\begin{theorem}\label{thm6.2}
Let the assumptions in Theorem \ref{thm6.1} hold, $\alpha\geq N-4$
(i.e., $\bar{p}\geq 2$) and $\alpha< N-2$. Then the set
$\mathcal{G}_a$ is orbitally stable.
\end{theorem}

To prove Theorems \ref{thm6.1} and \ref{thm6.2}, we follow the
strategy of \cite{JEANJEAN-JENDREJ}. In the proofs, a special role
will be played by the Poho\v{z}aev set
\begin{equation*}
\mathcal{P}_{a}:=\{u\in S_a: P(u)=0\},
\end{equation*}
where
\begin{equation*}
\begin{split}
P(u):=\int_{\mathbb{R}^N}|\nabla
u|^2-\int_{\mathbb{R}^N}(I_\alpha\ast|u|^{\bar{p}})|u|^{\bar{p}}-\mu\gamma_q\int_{\mathbb{R}^N}|u|^{q}.
\end{split}
\end{equation*}
The set $\mathcal{P}_{a}$ is quite related to the fiber map
\begin{equation}\label{e6.6}
\begin{split}
\Psi_u(\tau):=E(u_\tau)=\frac{1}{2}\tau^{2}\|\nabla
u\|_2^2-\frac{1}{2\bar{p}}\tau^{2\bar{p}}\int_{\mathbb{R}^N}
(I_\alpha\ast|u|^{\bar{p}})|u|^{\bar{p}}-\frac{\mu}{q}\tau^{q\gamma_q}\|u\|_q^{q},
\end{split}
\end{equation}
where
\begin{equation}\label{e6.5}
u_\tau(x):=\tau^{\frac{N}{2}}u(\tau x),\ x\in\mathbb{R}^N,\ \tau>0.
\end{equation}
The fiber map $\Psi_u(\tau)$ is introduced by Jeanjean in
\cite{Jeanjean 1997} for the Schr\"{o}dinger equation  and is well
studied by Soave in \cite{Soave JDE}. According to $\Psi_u(\tau)$,
$\mathcal{P}_{a}=\mathcal{P}_{a,+}\cup \mathcal{P}_{a,-}$, where
\begin{equation*}
\mathcal{P}_{a,+}:=\{u\in \mathcal{P}_{a}: E(u)<0\},\
\mathcal{P}_{a,-}:=\{u\in \mathcal{P}_{a}: E(u)> 0\}
\end{equation*}
if $\mu a^{\frac{q(1-\gamma_q)}{2}}\leq
(2K)^{\frac{q\gamma_q-2\bar{p}}{2(\bar{p}-1)}}$, see Lemmas
\ref{lem6.6} and \ref{lem7.21}.

The ground state $u_{+}$ obtained in Theorem \ref{thm6.1} lies on
$\mathcal{P}_{a,+}$ and can be characterized by
\begin{equation*}
E(u_{+})=\inf_{u\in \mathcal{P}_{a,+}}E(u)=\inf_{u\in
V_a}E(u)=m_{a}.
\end{equation*}
The critical point $u_{-}$ which will be  obtained in the following
theorem lies on $\mathcal{P}_{a,-}$ and can be characterized by
$E(u_{-})=\inf_{u\in \mathcal{P}_{a,-}}E(u)$. Precisely,

\begin{theorem}\label{thm6.3}
Let $N\geq 3$, $\alpha\in (0,N)$, $2<q<2+\frac{4}{N}$, $p=\bar{p}$,
$\mu>0$, $a>0$, $\mu a^{\frac{q(1-\gamma_q)}{2}}\leq
(2K)^{\frac{q\gamma_q-2\bar{p}}{2(\bar{p}-1)}}$. Then there exists a
second solution $u_{-}$ to (\ref{e1.5}) which satisfies
\begin{equation*}
0<E(u_{-})=\inf_{u\in
\mathcal{P}_{a,-}}E(u)<m_a+\frac{2+\alpha}{2(N+\alpha)}S_\alpha^{\frac{N+\alpha}{2+\alpha}}.
\end{equation*}
In particular, $u_{-}$  is not a ground state.
\end{theorem}

\begin{remark}
Note that the result is new in the case $\mu
a^{\frac{q(1-\gamma_q)}{2}}=(2K)^{\frac{q\gamma_q-2\bar{p}}{2(\bar{p}-1)}}$.
There is not corresponding result even to the Schr\"{o}dinger
equation (\ref{e7.28}). During the proof, the  lower bound of
$\inf_{u\in \mathcal{P}_{a,-}}E(u)$ obtained in  Lemma \ref{lem7.21}
plays an important role. The proof of Lemma \ref{lem7.21} is
interesting. Maybe it give us some insights to consider the case
$\mu
a^{\frac{q(1-\gamma_q)}{2}}>(2K)^{\frac{q\gamma_q-2\bar{p}}{2(\bar{p}-1)}}$.
\end{remark}

We combine the methods used in \cite{Jeanjean-Le} and \cite{Wei-Wu
2021} to prove Theorem \ref{thm6.3}. Precisely, we first use the
mountain pass lemma to obtain a Palais-Smale sequence $\{u_n\}$ of
$E$ on $S_{a}\cap H_r^1(\mathbb{R}^N)$ with $P(u_n)\to 0$ and
$E(u_n)\to M_r(a)$ as $n\to\infty$, see Lemma \ref{pro1.2}.
Secondly, by using the Poho\v{z}aev constraint method and the
Schwartz  rearrangement, we can show that
\begin{equation*}
M_{r}(a)=M(a)=\inf_{u\in
\mathcal{P}_{a,-}}E(u)=\inf_{\mathcal{P}_{a,-}\cap
H_r^1(\mathbb{R}^N)}E(u),
\end{equation*}
see Lemma \ref{pro1.3}. Thirdly, by using the radial symmetry of
$\{u_n\}$ and the bounds of $\inf_{u\in \mathcal{P}_{a,-}}E(u)$, we
can show that $\{u_n\}$ converges to a solution  to  (\ref{e1.5}).
In the proof, to obtain the upper bound of $\inf_{u\in
\mathcal{P}_{a,-}}E(u)$ is a difficult task. When $N\geq 5$ and
$\bar{p}<2$, the methods used in \cite{Wei-Wu 2021} can not threat
the nonlocal term $(I_\alpha\ast|u|^{\bar{p}})|u|^{\bar{p}-2}u$
directly, see Lemma \ref{lem7.1}.  Inspired by \cite{Jeanjean-Le},
by using the radially non-increasing of $u_+$ and by calculating the
nonlocal term carefully (see (\ref{e8.2})), we can choose
$\{y_\epsilon\}$ satisfying (\ref{e8.1}) and (\ref{e8.9}).  Based of
which, we can give the upper bound of $\inf_{u\in
\mathcal{P}_{a,-}}E(u)$ when $N\geq 5$ and $\bar{p}<2$, see Lemma
\ref{lem7.3} .

The following result is about the positivity, radial symmetry and
exponential decay  of the ``second class" solution.

\begin{theorem}\label{thm1.3}
Assume the conditions in Theorem \ref{thm6.3}  hold. Let $u$ be a
solution to (\ref{e1.5}) with $E(u)=\inf_{v\in
\mathcal{P}_{a,-}}E(v)$, then

(1) $|u|>0$;

(2) There exist $x_0\in \mathbb{R}^N$ and a non-increasing positive
function $v: (0,\infty)\to\mathbb{R}$ such that $|u(x)|=v(|x-x_0|)$
for almost every $x\in \mathbb{R}^N$;

(3) If $\alpha\geq N-4$ (i.e., $\bar{p}\geq 2$), then $|u|$ has
exponential decay at infinity:
\begin{equation*}
|u(x)|\leq C e^{-\delta |x|},\ \ |x|\geq r_0,
\end{equation*}
for some $C>0$, $\delta>0$ and  $r_0>0$.
\end{theorem}

The positivity is obtained by using the properties of $\Psi_u(\tau)$
and $\inf_{v\in \mathcal{P}_{a,-}}E(v)$. The symmetry is obtained by
using the theories of polarization and the fact that $\inf_{v\in
\mathcal{P}_{a,-}}E(v)$ is a mountain pass level value. This method
is motivated by \cite{Moroz-Schaftingen 2015}. The exponential decay
follows the  radial symmetry, the estimate of
$(I_\alpha\ast|u|^{\bar{p}})$ and the exponential decay studied in
\cite{Berestycki-Lions 1983} to the Schr\"{o}dinger equation.
Theorem \ref{thm1.3} plays an important role in proving the
following result.

Theorem \ref{thm1.4} is about the instability of the ``second class"
solution, which is very new in the existing researches. As we know,
most existing results are about the instability of a solution but
not all solutions.

\begin{theorem}\label{thm1.4}
Assume  the conditions in Theorem \ref{thm6.2}  hold. Let $u$ be a
solution to (\ref{e1.5}) with $E(u)=\inf_{v\in
\mathcal{P}_{a,-}}E(v)$, then $\lambda<0$ and the associated
standing wave $e^{-i\lambda t}u$ is strongly unstable.
\end{theorem}

\begin{remark}\label{rmk1.1}
The conditions $\alpha\geq N-4$ (i.e., $\bar{p}\geq 2$) and $\alpha<
N-2$ in Theorems \ref{thm6.2} and \ref{thm1.4} are added for
obtaining the local existence of solution to (\ref{e1.2}), see Lemma
\ref{lem5.2}. The condition $\alpha\geq N-4$ in Theorem \ref{thm1.4}
is also needed to prove the exponential decay of $u$ (see (3) in
Theorem \ref{thm1.3}) which is used to show that $|x|u\in
L^2(\mathbb{R}^N)$.

The condition $\bar{p}\geq 2$ is added since the nonlinearity
$(I_\alpha\ast|u|^{\bar{p}})|u|^{\bar{p}-2}u$ is singular when
$\bar{p}<2$. We do not know whether it can be removed. While, the
condition $\alpha< N-2$ is added for technical reason, and we guess
it can be removed.
\end{remark}

This paper is organized as follows. In Section 2, we cite some
preliminaries. Sections 3-5 are devoted to the proofs of Theorems
\ref{thm6.1}, \ref{thm6.3} and \ref{thm1.3}, respectively. In
Section 6, we first give a local existence result, and then prove
Theorems \ref{thm6.2} and \ref{thm1.4}.

\textbf{Notation}: In this paper, it is understood that all
functions, unless otherwise stated, are complex valued, but for
simplicity we write $L^r(\mathbb{R}^N)$, $W^{1,r}(\mathbb{R}^N)$,
$H^1(\mathbb{R}^N)$ $D^{1,2}(\mathbb{R}^N)$, .... For $1\leq
r<\infty$, $L^r(\mathbb{R}^N)$ is the usual Lebesgue space endowed
with the norm $\|u\|_r^r:=\int_{\mathbb{R}^N}|u|^r$,
$W^{1,r}(\mathbb{R}^N)$
 is the usual Sobolev space endowed with the
norm $\|u\|_{W^{1,r}}^r:=\|\nabla u\|_r^r+\|u\|_r^r$,
$H^1(\mathbb{R}^N)=W^{1,2}(\mathbb{R}^N)$ and
$\|u\|_{H^1}^2:=\|u\|_{W^{1,2}}^2$,
$D^{1,2}(\mathbb{R}^N):=\left\{u\in L^{2^*}(\mathbb{R}^N):|\nabla
u|\in L^{2}(\mathbb{R}^N)\right\}$. $H_{r}^1(\mathbb{R}^N)$ denotes
the subspace of functions in $H^1(\mathbb{R}^N)$ which are radially
symmetric with respect to zero.  $S_{a,r}:=S_a\cap
H_{r}^1(\mathbb{R}^N)$. $C,\ C_1,\ C_2,\ ...$ denote positive
constants, whose values can change from line to line. The notation
$A\lesssim B$ means that $A\leq CB$ for some constant $C> 0$. If
$A\lesssim B \lesssim A$, we write $A\thickapprox B$.

\section{Preliminaries}
\setcounter{section}{2} \setcounter{equation}{0}

The following Gagliardo-Nirenberg inequality can be found in
\cite{Weinstein 1983}.
\begin{lemma}\label{lem2.5}
Let $N\geq 1$ and $2<p<2^*$, then the following sharp
Gagliardo-Nirenberg inequality
\begin{equation*}
\|u\|_{p}\leq C_{N,p}\|u\|_2^{1-\gamma_p}\|\nabla u\|_2^{\gamma_p}
\end{equation*}
holds for any $u\in H^1(\mathbb{R}^N)$, where the sharp constant
$C_{N,p}$ is
\begin{equation*}
C_{N,p}^{p}=\frac{2p}{2N+(2-N)p}\left(\frac{2N+(2-N)p}{N(p-2)}\right)^{\frac{N(p-2)}{4}}\frac{1}{\|Q_{p}\|_2^{p-2}}
\end{equation*}
and $Q_p$ is the unique positive radial solution of equation
\begin{equation*}
-\Delta Q+Q=|Q|^{p-2}Q.
\end{equation*}
\end{lemma}

The following well-known Hardy-Littlewood-Sobolev inequality  can be
found in \cite{Lieb-Loss 2001}.

\begin{lemma}\label{lem HLS}
Let $N\geq 1$, $p$, $r>1$ and $0<\beta<N$ with
$1/p+(N-\beta)/N+1/r=2$. Let $u\in L^p(\mathbb{R}^N)$ and $v\in
L^r(\mathbb{R}^N)$. Then there exists a sharp constant
$C(N,\beta,p)$, independent of $u$ and $v$, such that
\begin{equation*}
\left|\int_{\mathbb{R}^N}\int_{\mathbb{R}^N}\frac{u(x)v(y)}{|x-y|^{N-\beta}}dxdy\right|\leq
C(N,\beta,p)\|u\|_p\|v\|_r.
\end{equation*}
If $p=r=\frac{2N}{N+\beta}$, then
\begin{equation*}
C(N,\beta,p)=C_\beta(N)=\pi^{\frac{N-\beta}{2}}\frac{\Gamma(\frac{\beta}{2})}{\Gamma(\frac{N+\beta}{2})}\left\{\frac{\Gamma(\frac{N}{2})}{\Gamma(N)}\right\}^{-\frac{\beta}{N}}.
\end{equation*}
\end{lemma}

\begin{remark}\label{rek1.31}
(1). By the Hardy-Littlewood-Sobolev inequality above, for any $v\in
L^s(\mathbb{R}^N)$ with $s\in(1,N/\alpha)$, $I_\alpha\ast v\in
L^{\frac{Ns}{N-\alpha s}}(\mathbb{R}^N)$ and
\begin{equation*}
\|I_\alpha\ast v\|_{L^{\frac{Ns}{N-\alpha s}}}\leq C\|v\|_{L^s},
\end{equation*}
where $C>0$ is a constant depending only on $N,\ \alpha$ and $s$.

(2). By the Hardy-Littlewood-Sobolev inequality above and the
Sobolev embedding theorem, we obtain
\begin{equation}\label{e22.4}
\begin{split}
\int_{\mathbb{R}^N}(I_\beta\ast|u|^p)|u|^p\leq
C\left(\int_{\mathbb{R}^N}|u|^{\frac{2Np}{N+\beta}}\right)^{1+\beta/N}
\leq C\|u\|_{H^1(\mathbb{R}^N)}^{2p}
\end{split}
\end{equation}
for any $p\in \left[1+\beta/N,(N+\beta)/(N-2)\right]$ if $N\geq 3$
and $p\in \left[1+\beta/N,+\infty\right)$ if $N=1, 2$, where $C>0$
is a constant depending only on $N,\ \beta$ and $p$.
\end{remark}

The following fact is used in this paper (see \cite{Li-JDE-2020}).

\begin{lemma}\label{lem33.3}
Let $N\geq 3$, $\alpha\in(0,N)$ and
$p\in\left[\frac{N+\alpha}{N},\frac{N+\alpha}{N-2}\right]$. Assume
that $\{w_n\}_{n=1}^{\infty}\subset H^1(\mathbb{R}^N)$ satisfying
$w_n\rightharpoonup w$ weakly in $H^1(\mathbb{R}^N)$ as
$n\to\infty$, then
\begin{equation*}
(I_\alpha\ast|w_n|^p)|w_n|^{p-2}w_n\rightharpoonup
(I_\alpha\ast|w|^p)|w|^{p-2}w\ \mathrm{weakly\ in\
}H^{-1}(\mathbb{R}^N)\ \mathrm{as}\ n\to\infty.
\end{equation*}
\end{lemma}

The following lemma is used in this paper,  see
\cite{Berestycki-Lions 1983} for its proof.

\begin{lemma}\label{lem jx}
Let $N\geq 3$ and $ 1\leq t<+\infty$. If $u\in L^t(\mathbb{R}^N)$ is
a radial non-increasing function (i.e. $0\leq u(x)\leq u(y)$ if
$|x|\geq |y|$), then one has
\begin{equation*}
|u(x)|\leq |x|^{-N/t}\left(\frac{N}{|S^{N-1}|}\right)^{1/t}\|u\|_t,
\ x\neq 0,
\end{equation*}
where $|S^{N-1}|$  is the area of the unit sphere in $\mathbb{R}^N$.
\end{lemma}

The following Poho\v{z}aev identity is cited from \cite{Li-Ma 2020},
where the proof is given for $\lambda>0$ but it clearly extends to
 $\lambda\in \mathbb{R}$.
\begin{lemma}\label{lem3.9}
Let $N\geq 3$, $\alpha\in (0,N)$, $\lambda\in \mathbb{R}$, $\mu\in
\mathbb{R}$, $p\in
\left[\frac{N+\alpha}{N},\frac{N+\alpha}{N-2}\right]$ and $q\in
[2,2^*]$. If $u\in H^1(\mathbb{R}^N)$ is a solution to (\ref{e1.4}),
then $u$ satisfies the Poho\v{z}aev identity
\begin{align*}
\frac{N-2}{2}\int_{\mathbb{R}^N}|\nabla u|^2
=\frac{N\lambda}{2}\int_{\mathbb{R}^N}|u|^2+\frac{N+\alpha}{2p}\int_{\mathbb{R}^N}(I_{\alpha}\ast|u|^{p})|u|^{p}+\frac{\mu
N}{q}\int_{\mathbb{R}^N}|u|^{q}.
\end{align*}
\end{lemma}

\begin{lemma}\label{lem3.8}
Assume the conditions in Lemma \ref{lem3.9} hold. If $u\in
H^1(\mathbb{R}^N)$ is a solution to (\ref{e1.4}), then $P(u)=0$.
\end{lemma}

\begin{proof}
Multiplying (\ref{e1.4}) by $u$ and integrating over $\mathbb{R}^N$,
we derive
\begin{equation*}
\int_{\mathbb{R}^N}|\nabla
u|^2=\lambda\int_{\mathbb{R}^N}|u|^2+\int_{\mathbb{R}^N}(I_{\alpha}\ast|u|^{p})|u|^{p}+\mu\int_{\mathbb{R}^N}|u|^{q},
\end{equation*}
which combines with the Poho\v{z}aev identity from Lemma
\ref{lem3.9} give that $P(u)=0$.
\end{proof}

\section{Existence of normalized ground state standing waves}

\setcounter{section}{3} \setcounter{equation}{0}

In this section, we prove Theorem \ref{thm6.1}. We first study the
lower bound of $E(u)$. By (\ref{e3.14}) and  Lemma \ref{lem2.5}, we
obtain,  for any $u\in S_a$,
\begin{equation}\label{e3.11}
\begin{split}
E(u) &\geq \frac{1}{2}\|\nabla
u\|_2^2-\frac{1}{2\bar{p}}S_\alpha^{-\bar{p}}\|\nabla
u\|_2^{2\bar{p}}-\frac{\mu}{q} C_{N,q}^qa^{q(1-\gamma_q)/2}\|\nabla
u\|_2^{q\gamma_q}\\
&=\|\nabla u\|_2^2f_{\mu,a}(\|\nabla u\|_2^2)
\end{split}
\end{equation}
with
\begin{equation}\label{e6.1}
f_{\mu,a}(\rho):=\frac{1}{2}-\frac{1}{2\bar{p}}S_\alpha^{-\bar{p}}\rho^{\bar{p}-1}
-\frac{\mu}{q}C_{N,q}^qa^{\frac{q(1-\gamma_q)}{2}}\rho^{\frac{q\gamma_q-2}{2}},\
\ \rho\in (0,\infty).
\end{equation}

Next we study the properties of $f_{\mu,a}(\rho)$.

\begin{lemma}\label{lem6.1}
Let $N\geq 3$, $\alpha\in (0,N)$, $\mu>0$, $a>0$, $p=\bar{p}$,
 $q\in (2,2+\frac{4}{N})$ and $K$ be
defined in (\ref{e1.14}). Then
\begin{equation}\label{e1.15}
\max_{\rho>0}f_{\mu,a}(\rho) \left\{\begin{array}{ll}
>0,&  \ \mathrm{if}\ \mu a^{\frac{q(1-\gamma_q)}{2}}<(2K)^{\frac{q\gamma_q-2\bar{p}}{2(\bar{p}-1)}},\\
=0,&  \ \mathrm{if}\ \mu a^{\frac{q(1-\gamma_q)}{2}}=(2K)^{\frac{q\gamma_q-2\bar{p}}{2(\bar{p}-1)}},\\
<0,&  \ \mathrm{if}\ \mu
a^{\frac{q(1-\gamma_q)}{2}}>(2K)^{\frac{q\gamma_q-2\bar{p}}{2(\bar{p}-1)}}.
\end{array}\right.
\end{equation}
\end{lemma}

\begin{proof}
By the definition of $f_{\mu,a}(\rho)$, we have that
\begin{equation*}
f_{\mu,a}'(\rho)=-\frac{\bar{p}-1}{2\bar{p}}S_\alpha^{-\bar{p}}\rho^{\bar{p}-2}
-\frac{\mu}{q}\frac{q\gamma_q-2}{2}C_{N,q}^qa^{\frac{q(1-\gamma_q)}{2}}\rho^{\frac{q\gamma_q-2}{2}-1}.
\end{equation*}
Hence, the equation $f_{\mu,a}'(\rho)=0$ has a unique solution given
by
\begin{equation}\label{e6.2}
\rho_{\mu,a}=\left(\frac{\bar{p}\mu(2-q\gamma_q)}{q(\bar{p}-1)}C_{N,q}^qa^{\frac{q(1-\gamma_q)}{2}}
S_\alpha^{\bar{p}}\right)^{\frac{2}{2\bar{p}-q\gamma_q}}.
\end{equation}
Taking into account that $f_{\mu,a}(\rho)\to -\infty$ as $\rho\to
0^+$ and $f_{\mu,a}(\rho)\to -\infty$ as $\rho\to +\infty$, we
obtain that $\rho_{\mu,a}$ is the unique global maximum point of
$f_{\mu,a}(\rho)$ and the maximum value is
\begin{equation*}
\max_{\rho>0}f_{\mu,a}(\rho)=f_{\mu,a}(\rho_{\mu,a})=\frac{1}{2}
-K\left[\mu
a^{\frac{q(1-\gamma_q)}{2}}\right]^{\frac{2(\bar{p}-1)}{2\bar{p}-q\gamma_q}},
\end{equation*}
which implies (\ref{e1.15}) holds.
\end{proof}

\begin{lemma}\label{lem6.2}
Let $N\geq 3$, $\alpha\in (0,N)$, $\mu>0$, $p=\bar{p}$ and $q\in
(2,2+\frac{4}{N})$. If $a_1>0$ and  $\rho_1>0$ are such that
$f_{\mu,a_1}(\rho_1)\geq 0$, then for any $a_2\in (0,a_1)$, we have
\begin{equation}\label{e6.4}
f_{\mu,a_2}(\rho_2)> 0\ \mathrm{for}\ \ \rho_2\in
\left[\frac{a_2}{a_1}\rho_1,\rho_1\right].
\end{equation}
\end{lemma}

\begin{proof}
It is obvious that $f_{\mu,a_2}(\rho_1)>f_{\mu,a_1}(\rho_1)\geq 0$,
and by direct calculation,
\begin{equation*}
\begin{split}
f_{\mu,a_2}\left(\frac{a_2}{a_1}\rho_1\right)& =\frac{1}{2}
-\frac{1}{2\bar{p}}S_\alpha^{-\bar{p}}\left(\frac{a_2}{a_1}\right)^{\bar{p}-1}\rho_1^{\bar{p}-1}
-\frac{\mu}{q}C_{N,q}^q\left(\frac{a_2}{a_1}\right)^{\frac{q}{2}-1}a_1^{\frac{q(1-\gamma_q)}{2}}\rho_1^{\frac{q\gamma_q-2}{2}}\\
&> \frac{1}{2}
-\frac{1}{2\bar{p}}S_\alpha^{-\bar{p}}\rho_1^{\bar{p}-1}
-\frac{\mu}{q}C_{N,q}^qa_1^{\frac{q(1-\gamma_q)}{2}}\rho_1^{\frac{q\gamma_q-2}{2}}=f_{\mu,a_1}(\rho_1)\geq
0.
\end{split}
\end{equation*}
It follows from the properties of  $f_{\mu,a_2}(\rho)$ studied in
Lemma \ref{lem6.1} that (\ref{e6.4}) holds.
\end{proof}

By Lemma \ref{lem6.1}, the domain $\{(\mu,a)\in \mathbb{R}^2:
\mu>0,\ a>0\}$ is divided into three parts $\Omega_1$, $\Omega_2$
and $\Omega_3$ by the curve $\mu
a^{\frac{q(1-\gamma_q)}{2}}=(2K)^{\frac{q\gamma_q-2\bar{p}}{2(\bar{p}-1)}}$
with
\begin{equation*}
\Omega_1=\{(\mu,a)\in \mathbb{R}^2: \mu>0,\ a>0, \ \mu
a^{\frac{q(1-\gamma_q)}{2}}<(2K)^{\frac{q\gamma_q-2\bar{p}}{2(\bar{p}-1)}}\},
\end{equation*}
\begin{equation*}
\Omega_2=\{(\mu,a)\in \mathbb{R}^2: \mu>0,\ a>0, \ \mu
a^{\frac{q(1-\gamma_q)}{2}}=(2K)^{\frac{q\gamma_q-2\bar{p}}{2(\bar{p}-1)}}\}
\end{equation*}
and
\begin{equation*}
\Omega_3=\{(\mu,a)\in \mathbb{R}^2: \mu>0,\ a>0, \ \mu
a^{\frac{q(1-\gamma_q)}{2}}>(2K)^{\frac{q\gamma_q-2\bar{p}}{2(\bar{p}-1)}}\}.
\end{equation*}
In this paper, we will consider the domain $\Omega_1\cup\Omega_2$.
For fixed $\mu>0$, define $a_0$ such that
\begin{equation}\label{e6.3}
\mu
a_0^{\frac{q(1-\gamma_q)}{2}}=(2K)^{\frac{q\gamma_q-2\bar{p}}{2(\bar{p}-1)}}.
\end{equation}
Then $\Omega_1\cup\Omega_2=\{(\mu,a)\in\mathbb{R}^2,\mu>0, 0<a\leq
a_0\}$. Note that $\rho_0$ defined in (\ref{e**}) is
$\rho_{\mu,a_0}$, and by Lemmas \ref{lem6.1} and \ref{lem6.2},
$f_{\mu,a_0}(\rho_0)=0$ and $f_{\mu,a}(\rho_0)>0$ for $a\in
(0,a_0)$. Hence $\inf_{u\in
\partial V_a}E(u)\geq 0$. Moreover, $V_a$ is a potential well, see
Lemma \ref{lem6.3}.

For future use, we study the properties of $\Psi_u(\tau)$ defined in
(\ref{e6.6}).

\begin{lemma}\label{lem6.6}
Let $N\geq 3$, $\alpha\in (0,N)$, $\mu>0$, $p=\bar{p}$,
$q\in(2,2+\frac{4}{N})$ and $a\in (0, a_0]$.Then for every $u\in
S_a$, the function $\Psi_u(\tau)$ has exactly two critical points
$\tau_u^+$ and $\tau_u^-$ with $0<\tau_u^+<\tau_u^-$. Moreover:

(1) $\tau_u^+$ is a local minimum point for $\Psi_u(\tau)$,
$E(u_{\tau_u^+})<0$ and $u_{\tau_u^+}\in V_a$.

(2) $\tau_u^-$ is a global maximum point for $\Psi_u(\tau)$,
$\Psi_u'(\tau)<0$ for $\tau>\tau_u^-$ and
\begin{equation*}
E(u_{\tau_u^-})\geq \inf_{u\in \partial V_a}E(u)\geq 0.
\end{equation*}
In particular, if $a\in (0, a_0)$, then $\inf_{u\in \partial
V_a}E(u)> 0$.

(3) $\Psi_u''(\tau_u^-)<0$ and the maps $u\in S_a\mapsto \tau_u^-\in
\mathbb{R}$ is of class $C^1$.
\end{lemma}

\begin{proof}
The proof can be done by modifying the proof of (\cite{Jeanjean-Le},
Lemma 2.4) in a trivial way. So we omit it.
\end{proof}

\begin{lemma}\label{lem6.3}
Let $N\geq 3$, $\alpha\in (0,N)$, $p=\bar{p}$, $q\in
(2,2+\frac{4}{N})$, $\mu>0$ and $a\in (0, a_0]$. Then

(1) $m_a = \inf_{u\in V_a}E(u) < 0\leq \inf_{u\in \partial
V_a}E(u)$.

(2) If $m_a$ is reached, then any ground state to (\ref{e1.5}) is
contained in $V_a$.
\end{lemma}

\begin{proof}
(1) In view of Lemma \ref{lem6.6}, we just need to prove $\inf_{u\in
V_a}E(u)<0$. For any fixed $u\in S_a$, let $u_\tau(x)$ and
$\Psi_u(\tau)$ be defined in (\ref{e6.5}) and (\ref{e6.6}),
respectively. It is obvious that $\|\nabla u_\tau\|_2^2\to 0$ and
$E(u_\tau)=\Psi_u(\tau)\to 0^-$ as $\tau\to 0^+$. Hence, we can
choose $\tau_0>0$ sufficiently small such that $u_{\tau_0}\in V_a$
and $E(u_{\tau_0})<0$.

(2) Let $u\in V_a$ be such that $E(u)=m_a$. By (1), $u$ is a
solution to (\ref{e1.5}). Let $v$ be any ground state to
(\ref{e1.5}). Then $E(v)\leq E(u)= m_a<0$, and by Lemma
\ref{lem3.8}, $P(v)=0$. Consequently, by Lemma \ref{lem6.6},
$\tau_v^+=1$ and $v=v_{\tau_v^+}\in V_a$.
\end{proof}

\begin{lemma}\label{lem6.4}
Let $N\geq 3$, $\alpha\in (0,N)$, $p=\bar{p}$, $q\in
(2,2+\frac{4}{N})$ and $\mu>0$. Then

(1) $a\in (0,a_0]\mapsto m_a$ is a continuous mapping.

(2) Let $a\in (0,a_0]$. We have for every $a_1\in (0, a) : m_a\leq
m_{a_1}+m_{a-a_1}$, and if $m_{a_1}$ or $m_{a-a_1}$ is reached then
the inequality is strict.
\end{lemma}

\begin{proof}
The proof can be done by modifying the proof of
(\cite{JEANJEAN-JENDREJ}, Lemma 2.6) in a trivial way. So we omit
it.
\end{proof}

The following result will both imply the existence of a ground state
to (\ref{e1.5}) and will be a crucial step to derive the orbital
stability of the set $\mathcal{G}_a$.

\begin{proposition}\label{pro6.1}
Let $N\geq 3$, $\alpha\in (0,N)$, $p=\bar{p}$, $q\in
(2,2+\frac{4}{N})$, $\mu>0$ and $a\in (0, a_0]$. If $\{u_n\}\subset
B_{\rho_0}$ is such that $\|u_n\|_2^2\to a$ and $E(u_n)\to m_a$
then,  up to translation, $u_n$ converges to $u\in V_a$ strongly in
$H^1(\mathbb{R}^N)$.
\end{proposition}

\begin{proof}
Since $\{u_n\}\subset B_{\rho_0}$ and $\|u_n\|_2^2\to a$, we obtain
that $\{u_n\}$ is bounded in $H^1(\mathbb{R}^N)$. We  claim that
\begin{align}\label{e6.7}
\liminf_{n\to \infty}\sup_{y\in
\mathbb{R}^N}\int_{B_1(y)}|u_n(x)|^2dx>0.
\end{align}
If it is false,  $\|u_n\|_q\to 0$ as $n\to \infty$  by Lions'
vanishing lemma, see (\cite{Willem 1996}, Lemma 1.21). By using
(\ref{e3.14}) and $\{u_n\}\subset B_{\rho_0}$, we obtain that
\begin{equation*}
\begin{split}
E(u_n)&=\frac{1}{2}\|\nabla
u_n\|_2^2-\frac{1}{2\bar{p}}\int_{\mathbb{R}^N}(I_\alpha\ast|u_n|^{\bar{p}})|u_n|^{\bar{p}}+o_n(1)\\
&\geq \frac{1}{2}\|\nabla
u_n\|_2^2-\frac{1}{2\bar{p}}S_\alpha^{-\bar{p}}\|\nabla
u_n\|_2^{2\bar{p}}+o_n(1)\\
&=\|\nabla
u_n\|_2^2\left(\frac{1}{2}-\frac{1}{2\bar{p}}S_\alpha^{-\bar{p}}\|\nabla
u_n\|_2^{2\bar{p}-2}\right)+o_n(1)\\
&\geq \|\nabla
u_n\|_2^2\left(\frac{1}{2}-\frac{1}{2\bar{p}}S_\alpha^{-\bar{p}}\rho_0^{\bar{p}-1}\right)+o_n(1).
\end{split}
\end{equation*}
Since $f_{\mu,a_0}(\rho_0)=0$, we have that
\begin{equation*}
\frac{1}{2}-\frac{1}{2\bar{p}}S_\alpha^{-\bar{p}}\rho_0^{\bar{p}-1}=
\frac{\mu}{q}C_{N,q}^qa_0^{\frac{q(1-\gamma_q)}{2}}\rho_0^{\frac{q\gamma_q-2}{2}}>0.
\end{equation*}
Consequently, $E(u_n)\geq o_n(1)$, which contradicts $E(u_n)\to
m_a<0$.

So (\ref{e6.7}) holds. Going if necessary to a subsequence, there
exists a sequence $\{y_n\}\subset \mathbb{R}^N$ such that,
$\tilde{u}_n(x):=u_n(x-y_n)\rightharpoonup u\in
H^1(\mathbb{R}^N)\setminus\{0\}$
 weakly in $H^1(\mathbb{R}^N)$. Set $v_n=\tilde{u}_n-u$. Then by the weak
 convergence and the Brezis-Lieb lemma, we know
\begin{equation*}
\|u_n\|_2^2=\|\tilde{u}_n\|_2^2=\|u\|_2^2+\|v_n\|_2^2+o_n(1),
\end{equation*}
\begin{equation*}
\|\nabla u_n\|_2^2=\|\nabla \tilde{u}_n\|_2^2=\|\nabla
u\|_2^2+\|\nabla v_n\|_2^2+o_n(1),
\end{equation*}
\begin{equation*}
\|u_n\|_q^q=\|\tilde{u}_n\|_q^q=\|u\|_q^q+\|v_n\|_q^q+o_n(1)
\end{equation*}
and
\begin{equation*}
\begin{split}
\int_{\mathbb{R}^N}(I_\alpha\ast
|u_n|^{\bar{p}})|u_n|^{\bar{p}}&=\int_{\mathbb{R}^N}(I_\alpha\ast
|\tilde{u}_n|^{\bar{p}})|\tilde{u}_n|^{\bar{p}}\\
&=\int_{\mathbb{R}^N}(I_\alpha\ast
|u|^{\bar{p}})|u|^{\bar{p}}+\int_{\mathbb{R}^N}(I_\alpha\ast
|v_n|^{\bar{p}})|v_n|^{\bar{p}}+o_n(1).
\end{split}
\end{equation*}
Consequently,
\begin{equation*}
E(u_n)=E(\tilde{u}_n)=E(u)+E(v_n)+o_n(1).
\end{equation*}

Next, by repeating word by word the proof of Theorem 2.5 in
\cite{JEANJEAN-JENDREJ},  we can show that $\|v_n\|_2^2\to 0$ and
$\|\nabla v_n\|_2^2\to 0$ as $n\to \infty$. Thus, $\tilde{u}_n\to
u\in V_a$ strongly in $H^1(\mathbb{R}^N)$.
\end{proof}

\textbf{Proof of Theorem \ref{thm6.1}}. (1), (2) and (3) follow from
Proposition \ref{pro6.1} and  Lemma \ref{lem6.3}. To prove (4), we
let $|\tilde{u}|^*$ denote the Schwartz rearrangement of
$|\tilde{u}|$. Then
\begin{equation*}
\||\tilde{u}|^*\|_2^2=\|\tilde{u}\|_2^2=a,\
\|\nabla|\tilde{u}|^*\|_2^2\leq \|\nabla|\tilde{u}|\|_2^2\leq
\|\nabla\tilde{u}\|_2^2<\rho_0,\
\||\tilde{u}|^*\|_q^q=\|\tilde{u}\|_q^q,
\end{equation*}
\begin{equation*}
\int_{\mathbb{R}^N}(I_\alpha\ast
||\tilde{u}|^*|^{\bar{p}})||\tilde{u}|^*|^{\bar{p}}\geq
\int_{\mathbb{R}^N}(I_\alpha\ast
|\tilde{u}|^{\bar{p}})|\tilde{u}|^{\bar{p}}.
\end{equation*}
These imply that $|\tilde{u}|^*\in V_a$ and $E(|\tilde{u}|^*)\leq
E(\tilde{u})=m_a$. By the definition of $m_a$, we know $m_a$ is
attained by the positive and radially symmetric non-increasing
function $|\tilde{u}|^*$. Lastly we prove (5). By using the equation
(\ref{e1.5}),  $P(u)=0$, $0<\gamma_q<1$ and $\mu>0$, we obtain
\begin{equation}\label{e***}
\lambda a=\|\nabla u\|_2^2-
\int_{\mathbb{R}^N}(I_\alpha\ast|u|^{\bar{p}})|u|^{\bar{p}}-\mu
\|u\|_q^{q}=\mu (\gamma_q-1) \|u\|_q^{q}<0,
\end{equation}
which implies $\lambda<0$. The proof is complete.

\section{Existence of mountain pass type  normalized standing waves}

\setcounter{section}{4} \setcounter{equation}{0}

In this section, we prove Theorem \ref{thm6.3}. Firstly, we  use the
mountain pass lemma to obtain a special Palais-Smale sequence. Now
we set
\begin{equation*}
M_{r}(a):=\inf_{g\in \Gamma_{r}(a)}\max_{t\in [0,\infty)}E(g(t)),
\end{equation*}
where
\begin{equation*}
\Gamma_{r}(a):=\{g\in C([0,\infty),S_{a,r}):g(0)\in
\mathcal{P}_{a,+}, \exists t_{g}\ s.t.\ g(t)\in E_{2m_{a}}\
\mathrm{for}\ t\geq t_{g}\}
\end{equation*}
with $E_{c}:=\{u\in H^1(\mathbb{R}^N): E(u)<c\}$. Then we have

\begin{lemma}\label{pro1.2}
Let $N\geq 3$, $\alpha\in (0,N)$, $p=\bar{p}$, $q\in
(2,2+\frac{4}{N})$, $\mu>0$, $a>0$, $\mu
a^{\frac{q(1-\gamma_q)}{2}}\leq
(2K)^{\frac{q\gamma_q-2\bar{p}}{2(\bar{p}-1)}}$. Then there exists a
Palais-Smale sequence $\{u_n\}\subset S_{a,r}$ for $E|_{S_a}$ at
level $M_{r}(a)$, with $P(u_n)\to 0$ as $n\to \infty$.
\end{lemma}

\begin{proof}
We follow the strategy introduced in \cite{Jeanjean 1997} and
consider the functional $\tilde{E}:\mathbb{R}_+\times
H^1(\mathbb{R}^N)\to \mathbb{R}$ defined by
\begin{equation*}
\tilde{E}(\tau,u):=E(u_\tau)=\Psi_u(\tau).
\end{equation*}
Define
\begin{equation*}
\tilde{M}_r(a):=\inf_{\tilde{g}\in \tilde{\Gamma}_r(a)}\max_{t\in
[0,\infty)}\tilde{E}(\tilde{g}(t)),
\end{equation*}
where
\begin{equation*}
\begin{split}
\tilde{\Gamma}_{r}(a):=\{\tilde{g}\in
C([0,\infty),\mathbb{R}_+\times S_{a,r}):\tilde{g}(0)\in
(1,\mathcal{P}_{a,+}), \exists t_{\tilde{g}}\ s.t.\  \tilde{g}(t)\in
(1,E_{2m_{a}}), t\geq t_{\tilde{g}}\}.
\end{split}
\end{equation*}
Similarly to Lemma 3.3 in \cite{Jeanjean-Le}, we can show that
$\tilde{M}_r(a)=M_r(a)$, and
\begin{equation*}
\tilde{M}_r(a)=\inf_{\tilde{g}\in \tilde{\Gamma}_r(a)}\max_{t\in
[0,\infty)}\tilde{E}(\tilde{g}(t))\geq
0>\max\{\tilde{E}(\tilde{g}(0)),\tilde{E}(\tilde{g}(t_{\tilde{g}}))\}.
\end{equation*}
Then repeating word by word the proof of Proposition 1.10 in
\cite{Jeanjean-Le}, we can obtain a Palais-Smale sequence
$\{u_n\}\subset S_{a,r}$ for $E|_{S_a}$ at level $M_{r}(a)$, with
$P(u_n)\to 0$ as $n\to \infty$.
\end{proof}

Next we study the value of  $M_r(a)$. For this aim, we set
\begin{equation*}
M(a):=\inf_{g\in \Gamma(a)}\max_{t\in [0,\infty)}E(g(t)),
\end{equation*}
where
\begin{equation}\label{e6.8}
\Gamma(a):=\{g\in C([0,\infty),S_{a}):g(0)\in V_a\cap E_{0}, \exists
t_{g}\ s.t.\ g(t)\in E_{2m_{a}}, t\geq t_{g}\}.
\end{equation}

\begin{lemma}\label{pro1.3}
Let $N\geq 3$, $\alpha\in (0,N)$, $p=\bar{p}$, $q\in
(2,2+\frac{4}{N})$,  $\mu>0$, $a>0$, $\mu
a^{\frac{q(1-\gamma_q)}{2}}\leq
(2K)^{\frac{q\gamma_q-2\bar{p}}{2(\bar{p}-1)}}$. Then
\begin{equation*}
M_{r}(a)=M(a)=\inf_{\mathcal{P}_{a,-}}E(u)=\inf_{\mathcal{P}_{a,-}\cap
H_r^1(\mathbb{R}^N)}E(u).
\end{equation*}
\end{lemma}

\begin{proof}
Obviously,  $M_{r}(a)\geq M(a)$.

For any $g(t)\in \Gamma(a)$, since $g(0)\in V_a$, $E(g(0))<0$ and
$E(g(t_g))< 2m_{a}<m_{a}$, by Lemma \ref{lem6.6}, we have
$\tau_{g(0)}^->1$ and $\tau_{g(t_g)}^-<1$. So by the continuity of
$g(t)$ and of $u\mapsto \tau_{u}^-$, we know that there exists $t_0$
such that $\tau_{g(t_0)}^-=1$, i.e., $g(t_0)\in \mathcal{P}_{a,-}$.
Thus
\begin{equation*}
M(a)=\inf_{g\in \Gamma(a)}\max_{t\in[0,\infty)}E(g(t))\geq
\inf_{g\in \Gamma(a)}E(g(t_0))\geq \inf_{u\in
\mathcal{P}_{a,-}}E(u).
\end{equation*}

For any $u\in \mathcal{P}_{a,-}$, let $|u|^*$ be the Schwartz
 rearrangement of $|u|$. Since $\||u|^*\|_t=\|u\|_t$
with $t\in[1,\infty)$, $\|\nabla (|u|^*)\|_2\leq \|\nabla u\|_2$ and
\begin{equation*}
\int_{\mathbb{R}^N}(I_\alpha\ast(|u|^*)^{\bar{p}})(|u|^*)^{\bar{p}}\geq
\int_{\mathbb{R}^N}(I_\alpha\ast|u|^{\bar{p}})|u|^{\bar{p}},
\end{equation*}
we obtain that $\Psi_{|u|^*}(\tau)\leq \Psi_u(\tau)$ for any
$\tau\in [0,\infty)$. Let $\tau_u^-$ be defined by Lemma
\ref{lem6.6}  such that $P(u_{\tau_u^-})=0$. Then
\begin{equation*}
E(u)=\Psi_u(1)=\Psi_u(\tau_u^-)\geq \Psi_u(\tau_{|u|^*}^-)\geq
\Psi_{|u|^*}(\tau_{|u|^*}^-).
\end{equation*}
Since $|u|^*_{\tau_{|u|^*}^-}\in \mathcal{P}_{a,-}\cap
H_r^1(\mathbb{R}^N)$, we have that
\begin{equation*}
E(u)\geq \inf_{\mathcal{P}_{a,-}\cap H_r^1(\mathbb{R}^N)}E(u).
\end{equation*}
By the arbitrariness of $u$, we obtain that
\begin{equation}\label{e3.23}
\inf_{\mathcal{P}_{a,-}}E(u)\geq \inf_{\mathcal{P}_{a,-}\cap
H_r^1(\mathbb{R}^N)}E(u).
\end{equation}

For any $u\in \mathcal{P}_{a,-}\cap H_r^1(\mathbb{R}^N)$, define
\begin{equation*}
g_u(t):=u_{t+\tau_u^+},
\end{equation*}
where $\tau_u^+$ is defined by Lemma \ref{lem6.6}. Then $g_u(t)\in
\Gamma_r(a)$ and
\begin{equation*}
E(u)=\max_{t\in[0,\infty)}E(g_u(t))\geq M_r(a),
\end{equation*}
which implies that $\inf_{\mathcal{P}_{a,-}\cap
H_r^1(\mathbb{R}^N)}E(u)\geq M_r(a)$. The proof is complete.
\end{proof}

When $\mu a^{\frac{q(1-\gamma_q)}{2}}<
(2K)^{\frac{q\gamma_q-2\bar{p}}{2(\bar{p}-1)}}$, the lower bound of
$\inf_{u\in \mathcal{P}_{a,-}}E(u)$ is already studied in Lemma
\ref{lem6.6}. For clarity, we restate it in the following lemma.

\begin{lemma}\label{lem7.23}
Let $N\geq 3$, $\alpha\in (0,N)$, $p=\bar{p}$, $2<q<2+\frac{4}{N}$,
$\mu>0$, $a>0$ and $\mu a^{\frac{q(1-\gamma_q)}{2}}<
(2K)^{\frac{q\gamma_q-2\bar{p}}{2(\bar{p}-1)}}$. Then $\inf_{u\in
\mathcal{P}_{a,-}}E(u)>0$.
\end{lemma}

When $\mu
a^{\frac{q(1-\gamma_q)}{2}}=(2K)^{\frac{q\gamma_q-2\bar{p}}{2(\bar{p}-1)}}$,
by Lemma \ref{lem6.6}, $\inf_{u\in \mathcal{P}_{a,-}}E(u)\geq 0$.
Now we prove that the strict inequality holds.

\begin{lemma}\label{lem7.21}
Let $N\geq 3$, $\alpha\in (0,N)$, $p=\bar{p}$, $2<q<2+\frac{4}{N}$,
$\mu>0$, $a>0$ and $\mu a^{\frac{q(1-\gamma_q)}{2}}=
(2K)^{\frac{q\gamma_q-2\bar{p}}{2(\bar{p}-1)}}$. Then $\inf_{u\in
\mathcal{P}_{a,-}}E(u)>0$.
\end{lemma}

\begin{proof}
Suppose by contradiction that $\inf_{u\in \mathcal{P}_{a,-}}E(u)=0$.
Since by Lemma \ref{pro1.3},
$\inf_{\mathcal{P}_{a,-}}E(u)=\inf_{\mathcal{P}_{a,-}\cap
H_{r}^1(\mathbb{R}^N)}E(u)$, there exists $\{u_n\}\subset S_a\cap
H_r^1(\mathbb{R}^N)$ such that $P(u_n)=0$ and $E(u_n)=A_n$, where
$A_n\to 0^+$ as $n\to \infty$. By using $E(u_n)=o_n(1)$,
$P(u_n)=A_n$, $\|u_n\|_2^2=a$, (\ref{e3.14}) and Lemma \ref{lem2.5},
we obtain that
\begin{equation}\label{e7.22}
\begin{cases}
\|\nabla u_n\|_2^2=\frac{2\bar{p}-q\gamma_q}{q(\bar{p}-1)}\mu
\|u_n\|_q^q+C_1A_n\leq \frac{2\bar{p}-q\gamma_q}{q(\bar{p}-1)}\mu
C_{N,q}^q a^{\frac{q}{2}(1-\gamma_q)}\|\nabla u_n\|_2^{q\gamma_q}+C_1A_n,\\
\|\nabla
u_n\|_2^2=\frac{2\bar{p}-q\gamma_q}{\bar{p}(2-q\gamma_q)}\int_{\mathbb{R}^N}(I_\alpha\ast|u_n|^{\bar{p}})|u_n|^{\bar{p}}-C_2A_n
\leq \frac{2\bar{p}-q\gamma_q}{\bar{p}(2-q\gamma_q)}
S_\alpha^{-\bar{p}}\|\nabla u_n\|_2^{2\bar{p}}-C_2A_n,
\end{cases}
\end{equation}
where $C_1$ and $C_2$ are some positive constants. Consequently,
$\liminf_{n\to\infty}\|\nabla u_n\|_2^2>0$ and
\begin{equation*}
\begin{cases}
\|\nabla u_n\|_2^{2(1-q\gamma_q/2)}\leq \left(\frac{2\bar{p}-q\gamma_q}{\bar{p}(2-q\gamma_q)S_\alpha^{\bar{p}}}\right)^{\frac{q\gamma_q-2}{2(\bar{p}-1)}}+o_n(1),\\
\|\nabla u_n\|_2^{2(\bar{p}-1)}\geq
\frac{\bar{p}(2-q\gamma_q)S_\alpha^{\bar{p}}}{2\bar{p}-q\gamma_q}+o_n(1),
\end{cases}
\end{equation*}
which implies that
\begin{equation*}
\begin{cases}
\|\nabla u_n\|_2^{2}\leq  \rho_0+o_n(1),\\
\|\nabla u_n\|_2^{2}\geq \rho_0+o_n(1).
\end{cases}
\end{equation*}
Hence, $\|\nabla u_n\|_2^2\to \rho_0$ as $n\to \infty$, which
combines with (\ref{e7.22}) give that
\begin{equation}\label{e7.23}
\begin{cases}
\|u_n\|_q^q\to C_{N,q}^q \|u_n\|_2^{q(1-\gamma_q)}\|\nabla u_n\|_2^{q\gamma_q},\\
\int_{\mathbb{R}^N}(I_\alpha\ast|u_n|^{\bar{p}})|u_n|^{\bar{p}}\to
S_\alpha^{-\bar{p}}\|\nabla u_n\|_2^{2\bar{p}}
\end{cases}
\end{equation}
as $n\to \infty$. That is, $\{u_n\}\subset H_r^1(\mathbb{R}^N)$ is a
minimizing sequence of
\begin{equation}\label{e7.24}
\begin{split}
\frac{1}{C_{N,q}^q}:&=\inf_{ u\in
H^{1}(\mathbb{R}^N)\setminus\{0\}}\frac{\|\nabla
u\|_2^{q\gamma_q}\|u\|_2^{q(1-\gamma_q)}}{\|u\|_q^q}
\end{split}
\end{equation}
and
\begin{equation}\label{e7.25}
\begin{split}
S_\alpha:&=\inf_{ u\in
D^{1,2}(\mathbb{R}^N)\setminus\{0\}}\frac{\|\nabla
u\|_2^2}{\left(\int_{\mathbb{R}^N}\left(I_{\alpha}\ast
|u|^{\bar{p}}\right)|u|^{\bar{p}}\right)^{{1}/{\bar{p}}}}.
\end{split}
\end{equation}
Since $\{u_n\}\subset H_r^1(\mathbb{R}^N)$ is bounded, there exists
$u_0\in H_r^1(\mathbb{R}^N)$ such that $u_n\rightharpoonup u_0$
weakly in $H^1(\mathbb{R}^N)$, $u_n\to u_0$ strongly in
$L^t(\mathbb{R}^N)$ with $t\in (2,2^*)$ and $u_n\to u_0$ a.e. in
$\mathbb{R}^N$. By the weak convergence, we have $\|u_0\|_2^2\leq
\|u_n\|_2^2$ and $\|\nabla u_0\|_2^2\leq \|\nabla u_n\|_2^2$.
Consequently, $u_0$ is a minimizer of (\ref{e7.24}) and $u_n\to u_0$
strongly in $H^1(\mathbb{R}^N)$. By Theorem B in \cite{Weinstein
1983}, $u_0$ is the ground state of the equation
\begin{equation}\label{e7.26}
\frac{(q-2)N}{4}\Delta
u-\left(1+\frac{(q-2)(2-N)}{4}\right)u+|u|^{q-2}u=0.
\end{equation}
By using (\ref{e7.25}) and $u_n\to u_0$ strongly in
$H^1(\mathbb{R}^N)$, we obtain that $u_0$ is a minimizer of
$S_\alpha$. So $u_0$ is of the form
\begin{equation}\label{e7.27}
u_0=C\left(\frac{b}{b^2+|x|^2}\right)^{\frac{N-2}{2}},
\end{equation}
where $C>0$ is a fixed constant and $b\in (0,\infty)$ is a
parameter, see \cite{Gao-Yang-1}. (\ref{e7.27}) contradicts to
(\ref{e7.26}). Thus, $\inf_{u\in \mathcal{P}_{a,-}}E(u)>0$.
\end{proof}

The next two lemmas are about the upper bound of $\inf_{u\in
\mathcal{P}_{a,-}}E(u)$.

\begin{lemma}\label{lem7.1}
Let $N\geq 3$, $\alpha\in (0,N)$, $p=\bar{p}$, $q\in
(2,2+\frac{4}{N})$, $\mu>0$, $a>0$ and $\mu
a^{\frac{q(1-\gamma_q)}{2}}\leq
(2K)^{\frac{q\gamma_q-2\bar{p}}{2(\bar{p}-1)}}$. If $N\geq 5$, we
further assume that $\bar{p}\geq 2$ (i.e., $\alpha\geq N-4$). Then
\begin{equation*}
\inf_{u\in
\mathcal{P}_{a,-}}E(u)<m_a+\frac{2+\alpha}{2(N+\alpha)}S_\alpha^{\frac{N+\alpha}{2+\alpha}}.
\end{equation*}
\end{lemma}

\begin{proof}
For any $\epsilon>0$, we define
\begin{equation}\label{e8.10}
u_\epsilon(x)=\varphi(x)U_\epsilon(x),
\end{equation}
where $\varphi(x) \in C_c^{\infty}(\mathbb{R}^N)$ is a cut off
function satisfying: (a) $0\leq \varphi(x)\leq 1$ for any $x\in
\mathbb{R}^N$; (b) $\varphi(x)\equiv 1$ in $B_1$; (c)
$\varphi(x)\equiv 0$ in $\mathbb{R}^N\setminus \overline{B_2}$.
Here, $B_s$ denotes the ball in $\mathbb{R}^N$ of center at origin
and radius $s$.
\begin{equation*}
U_\epsilon(x)=\frac{\left(N(N-2)\epsilon^2\right)^{\frac{N-2}{4}}}{\left(\epsilon^2+|x|^2\right)^{\frac{N-2}{2}}},
\end{equation*}
where $U_1(x)$ is the extremal function of the minimizing problem
(\ref{e3.14}). In \cite{Gao-Yang-1}, they proved that
$S_\alpha=\frac{S}{(A_\alpha(N) C_\alpha(N))^{{1}/{\bar{p}}}}$,
where $A_\alpha(N)$ is defined in (\ref{e1.3}), $C_\alpha(N)$ is in
Lemma \ref{lem HLS} and
\begin{equation*}
S:=\inf_{ u\in
D^{1,2}(\mathbb{R}^N)\setminus\{0\}}\frac{\int_{\mathbb{R}^N}|\nabla
u|^2}{\left(\int_{\mathbb{R}^N}|u|^{\frac{2N}{N-2}}\right)^{\frac{N-2}{N}}}.
\end{equation*}

By \cite{Brezis-Nirenberg 1983} (see also \cite{Willem 1996}), we
have the following estimates.
\begin{equation}\label{e7.2}
\int_{\mathbb{R}^N}|\nabla
u_\epsilon|^2=S^{\frac{N}{2}}+O(\epsilon^{N-2}),\ N\geq 3,
\end{equation}
and
\begin{equation}\label{e7.3}
\int_{\mathbb{R}^N}| u_\epsilon|^2=\left\{\begin{array}{ll}
K_2\epsilon^2+O(\epsilon^{N-2}),& N\geq 5,\\
K_2\epsilon^2|\ln \epsilon|+O(\epsilon^2),& N=4,\\
K_2\epsilon+O(\epsilon^2),& N=3,
\end{array}\right.
\end{equation}
where $K_2>0$. By direct calculation, for  $t\in (2,2^*)$, there
exists $K_1>0$ such that
\begin{equation}\label{e7.4}
\begin{split}
\int_{\mathbb{R}^N}|u_\epsilon|^t &\geq
(N(N-2))^{\frac{N-2}{4}t}\epsilon^{N-\frac{N-2}{2}t}\int_{B_{\frac{1}{\epsilon}}(0)}\frac{1}{(1+|x|^2)^{\frac{N-2}{2}t}}dx\\
&\geq \left\{\begin{array}{ll}
K_1\epsilon^{N-\frac{N-2}{2}t},& (N-2)t>N,\\
K_1\epsilon^{N-\frac{N-2}{2}t}|\ln \epsilon|,& (N-2)t=N,\\
K_1\epsilon^{\frac{N-2}{2}t},& (N-2)t<N.
\end{array}\right.
\end{split}
\end{equation}
Moreover, similarly as in \cite{Gao-Yang-2} and \cite{Gao-Yang-1},
by direct computation, we have
\begin{equation}\label{e7.5}
\int_{\mathbb{R}^N}\left(I_\alpha\ast|u_\epsilon|^{\bar{p}}\right)
|u_\epsilon|^{\bar{p}}\geq (A_\alpha(N)
C_\alpha(N))^{\frac{N}2}S_\alpha^{\frac{N+\alpha}2}
+O(\epsilon^{\frac{N+\alpha}{2}}).
\end{equation}

Let $u_+$ be a positive and radially symmetric non-increasing ground
state to (\ref{e1.5}). For $t\geq 0$, we define
\begin{equation}\label{e8.11}
\hat{u}_{\epsilon,t}=u_++tu_{\epsilon}\ \mathrm{and}\
\bar{u}_{\epsilon,t}=\left(a^{-\frac{1}{2}}\|\hat{u}_{\epsilon,t}\|_2\right)^{\frac{N-2}{2}}
\hat{u}_{\epsilon,t}\left(a^{-\frac{1}{2}}\|\hat{u}_{\epsilon,t}\|_2x\right).
\end{equation}
Then
\begin{equation*}
\int_{\mathbb{R}^N}|\bar{u}_{\epsilon,t}|^2=a,\
\int_{\mathbb{R}^N}|\nabla\bar{u}_{\epsilon,t}|^2=\int_{\mathbb{R}^N}|\nabla\hat{u}_{\epsilon,t}|^2,
\end{equation*}
\begin{equation*}
\int_{\mathbb{R}^N}(I_\alpha\ast|\bar{u}_{\epsilon,t}|^{\bar{p}})|\bar{u}_{\epsilon,t}|^{\bar{p}}=
\int_{\mathbb{R}^N}(I_\alpha\ast|\hat{u}_{\epsilon,t}|^{\bar{p}})|\hat{u}_{\epsilon,t}|^{\bar{p}},
\end{equation*}
\begin{equation*}
\int_{\mathbb{R}^N}|\bar{u}_{\epsilon,t}|^{q}=\left(a^{-\frac{1}{2}}\|\hat{u}_{\epsilon,t}\|_2\right)^{q\gamma_q-q}
\int_{\mathbb{R}^N}|\hat{u}_{\epsilon,t}|^{q}.
\end{equation*}
Since $\bar{u}_{\epsilon,t}\in S_a$, by Lemma \ref{lem6.6}, there
exists a unique $\tau_{\epsilon,t}^->0$ such that
$(\bar{u}_{\epsilon,t})_{\tau_{\epsilon,t}^-}\in \mathcal{P}_{a,-}$,
which implies that
\begin{equation}\label{e7.1}
(\tau_{\epsilon,t}^-)^{2-q\gamma_q}\|\nabla\bar{u}_{\epsilon,t}\|_2^2
=(\tau_{\epsilon,t}^-)^{2\bar{p}-q\gamma_q}
\int_{\mathbb{R}^N}(I_\alpha\ast|\bar{u}_{\epsilon,t}|^{\bar{p}})|\bar{u}_{\epsilon,t}|^{\bar{p}}
+\mu\gamma_q\|\bar{u}_{\epsilon,t}\|_q^q.
\end{equation}
Since $\bar{u}_{\epsilon,0}=u_+\in \mathcal{P}_{a,+}$, by Lemma
\ref{lem6.6}, $\tau_{\epsilon,0}^->1$. By (\ref{e7.2}), (\ref{e7.5})
and (\ref{e7.1}), $\tau_{\epsilon,t}^-\to 0$ as $t\to +\infty$
uniformly for $\epsilon>0$ sufficiently small. Since
$\tau_{\epsilon,t}^-$ is unique by Lemma \ref{lem6.6}, it is
standard to show that $\tau_{\epsilon,t}^-$ is continuous for $t\geq
0$, which implies that there exists $t_\epsilon>0$ such that
$\tau_{\epsilon,t_\epsilon}^-=1$. Consequently, $\inf_{u\in
\mathcal{P}_{a,-}}E(u)\leq \sup_{t\geq 0}E(\bar{u}_{\epsilon,t})$
for any $\epsilon$ small enough. By (\ref{e7.2})-(\ref{e7.5}), and
the expression
\begin{equation}\label{e7.7}
\begin{split}
E(\bar{u}_{\epsilon,t})&=\frac{1}{2}\|\nabla
\hat{u}_{\epsilon,t}\|_2^2-\frac{1}{2\bar{p}}
\int_{\mathbb{R}^N}(I_\alpha\ast|\hat{u}_{\epsilon,t}|^{\bar{p}})|\hat{u}_{\epsilon,t}|^{\bar{p}}
-\frac{\mu}{q}\left(a^{-\frac{1}{2}}\|\hat{u}_{\epsilon,t}\|_2\right)^{q\gamma_q-q}\|\hat{u}_{\epsilon,t}\|_q^q,
\end{split}
\end{equation}
we have  $E(\bar{u}_{\epsilon,t})\to m_a$ as $t\to 0$, and
\begin{equation*}
\begin{split}
E(\bar{u}_{\epsilon,t})&\leq  t^2 \|\nabla
u_{\epsilon}\|_2^2-\frac{1}{2\bar{p}}
t^{2\bar{p}}\int_{\mathbb{R}^N}(I_\alpha\ast|u_{\epsilon}|^{\bar{p}})|u_{\epsilon}|^{\bar{p}}
\to -\infty
\end{split}
\end{equation*}
as $t\to +\infty$  uniformly for $\epsilon>0$ sufficiently small.
Hence, there exists $t_0>0$ large enough and $\epsilon_0>0$ small
enough such that
\begin{equation*}
E(\bar{u}_{\epsilon,t})<m_a+\frac{2+\alpha}{2(N+\alpha)}S_\alpha^{\frac{N+\alpha}{2+\alpha}}
\end{equation*}
for $t<\frac{1}{t_0}$ and $t>t_0$ uniformly for
$0<\epsilon<\epsilon_0$.

Next we  estimate  $E(\bar{u}_{\epsilon,t})$ for
$\frac{1}{t_0}<t<t_0$. By using the inequalities
\begin{equation*}
(a+b)^r\geq a^r+ra^{r-1}b+b^r,\ a>0,b>0, r\geq 2,
\end{equation*}
and
\begin{equation*}
(a+b)^r\geq a^r+ra^{r-1}b+rab^{r-1}+b^r,\ a>0,b>0, r\geq 3,
\end{equation*}
we obtain that
\begin{equation}\label{e7.8}
\|\nabla \hat{u}_{\epsilon,t}\|_2^2=\|\nabla
u_+\|_2^2+2t\int_{\mathbb{R}^N}\nabla u_+\nabla u_{\epsilon}+
\|\nabla (tu_{\epsilon})\|_2^2,
\end{equation}
\begin{equation}\label{e7.9}
\|\hat{u}_{\epsilon,t}\|_q^q\geq \|u_+\|_q^q+\|tu_{\epsilon}\|_q^q
+qt\int_{\mathbb{R}^N} |u_+|^{q-1}u_{\epsilon},
\end{equation}
\begin{equation*}
\|\hat{u}_{\epsilon,t}\|_2^2=\|u_+\|_2^2+2t\int_{\mathbb{R}^N} u_+
u_{\epsilon}+ \|tu_{\epsilon}\|_2^2,
\end{equation*}
\begin{equation}\label{e*}
\left(a^{-\frac{1}{2}}\|\hat{u}_{\epsilon,t}\|_2\right)^2
=1+\frac{2t}{a}\int_{\mathbb{R}^N} u_+
u_{\epsilon}+\frac{t^2}{a}\|u_{\epsilon}\|_2^2,
\end{equation}
\begin{equation}\label{e7.10}
\begin{split}
\int_{\mathbb{R}^N}&(I_\alpha\ast|\hat{u}_{\epsilon,t}|^{\bar{p}})|\hat{u}_{\epsilon,t}|^{\bar{p}}\\
&\geq
\int_{\mathbb{R}^N}(I_\alpha\ast|u_+|^{\bar{p}})|u_+|^{\bar{p}}
+2\bar{p}t\int_{\mathbb{R}^N}(I_\alpha\ast|u_+|^{\bar{p}})|u_+|^{\bar{p}-1}u_{\epsilon}\\
&\qquad
+2\bar{p}\int_{\mathbb{R}^N}(I_\alpha\ast|tu_{\epsilon}|^{\bar{p}})|tu_{\epsilon}|^{\bar{p}-1}u_{+}
+\int_{\mathbb{R}^N}(I_\alpha\ast|tu_{\epsilon}|^{\bar{p}})|tu_{\epsilon}|^{\bar{p}}
\end{split}
\end{equation}
for $N=3$, and
\begin{equation}\label{e7.11}
\begin{split}
\int_{\mathbb{R}^N}(I_\alpha\ast|\hat{u}_{\epsilon,t}|^{\bar{p}})|\hat{u}_{\epsilon,t}|^{\bar{p}}&\geq
\int_{\mathbb{R}^N}(I_\alpha\ast|u_+|^{\bar{p}})|u_+|^{\bar{p}} +
\int_{\mathbb{R}^N}(I_\alpha\ast|tu_{\epsilon}|^{\bar{p}})|tu_{\epsilon}|^{\bar{p}}\\
&\qquad+2\bar{p}t\int_{\mathbb{R}^N}(I_\alpha\ast|u_+|^{\bar{p}})|u_+|^{\bar{p}-1}u_{\epsilon}
\end{split}
\end{equation}
for $N\geq 4$, and  $N\geq 5$, $\bar{p}\geq 2$.

By the positivity of $u_+$, we have
\begin{equation}\label{e7.6}
\begin{split}
\int_{\mathbb{R}^N}u_+u_{\epsilon}
&\thickapprox \int_{B_1}\varphi(x)U_{\epsilon}(x)\\
&\thickapprox\epsilon^{\frac{N+2}{2}}\int_0^{\frac{1}{\epsilon}}\frac{1}{(1+r^2)^{\frac{N-2}{2}}}r^{N-1}dr\\
&\thickapprox
\epsilon^{\frac{N+2}{2}}\left(\frac{1}{\epsilon}\right)^2\thickapprox
\epsilon^{\frac{N-2}{2}}.
\end{split}
\end{equation}
By (\ref{e7.3}), (\ref{e*}), (\ref{e7.6}), and the inequality
$(1+t)^a\geq 1+at$ for $t\geq 0$ and $a<0$, we obtain that
\begin{equation}\label{e7.12}
\begin{split}
\left(a^{-\frac{1}{2}}\|\hat{u}_{\epsilon,t}\|_2\right)^{q\gamma_q-q}
&=\left(1+\frac{2t}{a}\int_{\mathbb{R}^N} u_+
u_{\epsilon}+\frac{t^2}{a}\|u_{\epsilon}\|_2^2\right)^{\frac{q\gamma_q-q}{2}}\\
&\geq 1+\frac{q\gamma_q-q}{2}\left(\frac{2t}{a}\int_{\mathbb{R}^N}
u_+ u_{\epsilon}+\frac{t^2}{a}\|u_{\epsilon}\|_2^2\right).
\end{split}
\end{equation}

Case $N=3$. Noting that $u_+$ satisfies the equation
\begin{equation*}
-\Delta u_+=\lambda
u_++(I_\alpha\ast|u_+|^{\bar{p}})|u_+|^{\bar{p}-2}u_++\mu|u_+|^{q-2}u_+
\end{equation*}
with $\lambda<0$ and $\lambda a=\mu(\gamma_q-1)\|u_+\|_q^q$ (see
(\ref{e***})), and by using (\ref{e7.7}), (\ref{e7.8}),
(\ref{e7.9}), (\ref{e7.10}), (\ref{e7.12}), we obtain that
\begin{equation}\label{e7.15}
\begin{split}
E(\bar{u}_{\epsilon,t})&\leq \frac{1}{2}\|\nabla
u_+\|_2^2+t\int_{\mathbb{R}^N}\nabla u_+\nabla
u_{\epsilon}+\frac{1}{2} \|\nabla (tu_{\epsilon})\|_2^2\\
&\quad
-\frac{1}{2\bar{p}}\int_{\mathbb{R}^N}(I_\alpha\ast|u_+|^{\bar{p}})|u_+|^{\bar{p}}
-t\int_{\mathbb{R}^N}(I_\alpha\ast|u_+|^{\bar{p}})|u_+|^{\bar{p}-1}u_{\epsilon}\\
&\quad
-\int_{\mathbb{R}^N}(I_\alpha\ast|tu_{\epsilon}|^{\bar{p}})|tu_{\epsilon}|^{\bar{p}-1}u_{+}
-\frac{1}{2\bar{p}}\int_{\mathbb{R}^N}(I_\alpha\ast|tu_{\epsilon}|^{\bar{p}})|tu_{\epsilon}|^{\bar{p}}\\
&\quad -\frac{\mu}{q}\|u_+\|_q^q-\frac{\mu}{q}\|tu_{\epsilon}\|_q^q
-\mu t\int_{\mathbb{R}^N} |u_+|^{q-1}u_{\epsilon}\\
&\quad
-\frac{\mu(\gamma_q-1)}{2}\left(\frac{2t}{a}\int_{\mathbb{R}^N} u_+
u_{\epsilon}+\frac{t^2}{a}\|u_{\epsilon}\|_2^2\right)\|\hat{u}_{\epsilon,t}\|_q^q\\
&=E(u_+)+E(tu_{\epsilon})+t\lambda\int_{\mathbb{R}^N}u_+u_{\epsilon}-\frac{\mu
t(\gamma_q-1)}{a}\|\hat{u}_{\epsilon,t}\|_q^q\int_{\mathbb{R}^N} u_+ u_{\epsilon}\\
&\quad-\frac{\mu(\gamma_q-1)t^2}{2a}\|u_{\epsilon}\|_2^2\|\hat{u}_{\epsilon,t}\|_q^q
-\int_{\mathbb{R}^N}(I_\alpha\ast|tu_{\epsilon}|^{\bar{p}})|tu_{\epsilon}|^{\bar{p}-1}u_{+}\\
&=m_a+E(tu_{\epsilon})+\frac{\mu
t(1-\gamma_q)}{a}(\|\hat{u}_{\epsilon,t}\|_q^q-\|u_{+}\|_q^q)\int_{\mathbb{R}^N} u_+ u_{\epsilon}\\
&\quad+\frac{\mu(1-\gamma_q)t^2}{2a}\|u_{\epsilon}\|_2^2\|\hat{u}_{\epsilon,t}\|_q^q
-\int_{\mathbb{R}^N}(I_\alpha\ast|tu_{\epsilon}|^{\bar{p}})|tu_{\epsilon}|^{\bar{p}-1}u_{+}.
\end{split}
\end{equation}
By direct calculation, we have
\begin{equation}\label{e7.14}
\begin{split}
\int_{\mathbb{R}^N}&(I_\alpha\ast|u_{\epsilon}|^{\bar{p}})|u_{\epsilon}|^{\bar{p}-1}u_{+}\\
&\gtrsim\int_{\mathbb{R}^N}(I_\alpha\ast|u_{\epsilon}|^{\bar{p}})|u_{\epsilon}|^{\bar{p}-1}\\
&\gtrsim
\int_{B_1}\int_{B_1}\frac{|U_{\epsilon}(x)|^{\bar{p}}|U_{\epsilon}(y)|^{\bar{p}-1}}{|x-y|^{N-\alpha}}dxdy\\
&=\epsilon^{\frac{N-2}{2}}\int_{B_{\frac{1}{\epsilon}}}\int_{B_{\frac{1}{\epsilon}}}
\frac{1}{(1+|x|^2)^{\frac{N-2}{2}\bar{p}}|x-y|^{N-\alpha}(1+|y|^2)^{\frac{N-2}{2}(\bar{p}-1)}}dxdy\\
&\gtrsim \epsilon^{\frac{N-2}{2}},
\end{split}
\end{equation}
\begin{equation}\label{e7.16}
\begin{split}
\|\hat{u}_{\epsilon,t}\|_q^q-\|u_{+}\|_q^q=\|u_++tu_{\epsilon}\|_q^q-\|u_{+}\|_q^q\lesssim
\int_{\mathbb{R}^N}|u_+|^{q-1}tu_{\epsilon}+\|tu_{\epsilon}\|_q^q,
\end{split}
\end{equation}
and similarly to (\ref{e7.6}),
\begin{equation}\label{e7.13}
\begin{split}
\int_{\mathbb{R}^N}u_+^{q-1}u_{\epsilon} \lesssim
\int_{B_2}U_{\epsilon}(x) \lesssim \epsilon^{\frac{N-2}{2}}.
\end{split}
\end{equation}
By (\ref{e7.2}), (\ref{e7.3}), (\ref{e7.5}), (\ref{e7.6}),
(\ref{e7.15}), (\ref{e7.14}), (\ref{e7.16}), (\ref{e7.13}),   we
obtain
\begin{equation*}
\begin{split}
E(\bar{u}_{\epsilon,t})&\leq
m_a+\frac{t^2}{2}\left(S^{\frac{N}{2}}+O(\epsilon^{N-2})\right)
-\frac{t^{2\bar{p}}}{2\bar{p}}\left((A_\alpha(N)
C_\alpha(N))^{\frac{N}2}S_\alpha^{\frac{N+\alpha}2}
+O(\epsilon^{\frac{N+\alpha}{2}})\right)\\
&\qquad -\frac{\mu}{q}t^q\|u_{\epsilon}\|_q^q+O(\epsilon^{N-2})
+O(\epsilon^{\frac{N-2}{2}})\|u_{\epsilon}\|_q^q+O(\|u_{\epsilon}\|_2^2)-C\epsilon^{\frac{N-2}{2}}\\
&<
m_a+\frac{t^2}{2}S^{\frac{N}{2}}-\frac{t^{2\bar{p}}}{2\bar{p}}(A_\alpha(N)
C_\alpha(N))^{\frac{N}2}S_\alpha^{\frac{N+\alpha}2}\\
&\leq
m_a+\frac{2+\alpha}{2(N+\alpha)}S_\alpha^{\frac{N+\alpha}{2+\alpha}}
\end{split}
\end{equation*}
for $\frac{1}{t_0}<t<t_0$ uniformly for $\epsilon\in (0,\epsilon_0)$
small enough.

Case $N\geq 4$. Similarly to case $N=3$, by using (\ref{e7.7}),
(\ref{e7.8}), (\ref{e7.9}), (\ref{e7.11}), (\ref{e7.12}), we have
\begin{equation}\label{e7.17}
\begin{split}
E(\bar{u}_{\epsilon,t})&\leq m_a+E(tu_{\epsilon})+\frac{\mu
t(1-\gamma_q)}{a}(\|\hat{u}_{\epsilon,t}\|_q^q-\|u_{+}\|_q^q)\int_{\mathbb{R}^N} u_+ u_{\epsilon}\\
&\quad+\frac{\mu(1-\gamma_q)t^2}{2a}\|u_{\epsilon}\|_2^2\|\hat{u}_{\epsilon,t}\|_q^q.
\end{split}
\end{equation}
Thus, by using (\ref{e7.2}), (\ref{e7.3}), (\ref{e7.4}),
(\ref{e7.5}), (\ref{e7.6}), (\ref{e7.16}),  (\ref{e7.13}) and
(\ref{e7.17}),
 we obtain
\begin{equation}\label{e9.1}
\begin{split}
E(\bar{u}_{\epsilon,t})&\leq
m_a+\frac{t^2}{2}\left(S^{\frac{N}{2}}+O(\epsilon^{N-2})\right)
-\frac{t^{2\bar{p}}}{2\bar{p}}\left((A_\alpha(N)
C_\alpha(N))^{\frac{N}2}S_\alpha^{\frac{N+\alpha}2}
+O(\epsilon^{\frac{N+\alpha}{2}})\right)\\
&\qquad -\frac{\mu}{q}t^q\|u_{\epsilon}\|_q^q+O(\epsilon^{N-2})
+O(\epsilon^{\frac{N-2}{2}})\|u_{\epsilon}\|_q^q+O(\|u_{\epsilon}\|_2^2)\\
&<
m_a+\frac{t^2}{2}S^{\frac{N}{2}}-\frac{t^{2\bar{p}}}{2\bar{p}}(A_\alpha(N)
C_\alpha(N))^{\frac{N}2}S_\alpha^{\frac{N+\alpha}2}\\
&\leq
m_a+\frac{2+\alpha}{2(N+\alpha)}S_\alpha^{\frac{N+\alpha}{2+\alpha}}
\end{split}
\end{equation}
for $\frac{1}{t_0}<t<t_0$ uniformly for $\epsilon\in (0,\epsilon_0)$
small enough. The proof is complete.
\end{proof}

\begin{lemma}\label{lem7.3}
Let $N\geq 4$, $\alpha\in (0,N)$, $p=\bar{p}$, $q\in
(2,2+\frac{4}{N})$, $\mu>0$, $a>0$ and $\mu
a^{\frac{q(1-\gamma_q)}{2}}\leq
(2K)^{\frac{q\gamma_q-2\bar{p}}{2(\bar{p}-1)}}$. Then
\begin{equation*}
\inf_{u\in
\mathcal{P}_{a,-}}E(u)<m_a+\frac{2+\alpha}{2(N+\alpha)}S_\alpha^{\frac{N+\alpha}{2+\alpha}}.
\end{equation*}
\end{lemma}

\begin{proof}
\textbf{Step 1}. Let $u_{+}$ and $u_{\epsilon}$ be defined  in Lemma
\ref{lem7.1}. We claim that for any $\epsilon>0$, there exists
$y_\epsilon\in \mathbb{R}^N$ such that
\begin{equation}\label{e8.1}
\int_{\mathbb{R}^N}u_{+}(x-y_{\epsilon})u_{\epsilon}(x)dx\leq
\|u_{\epsilon}\|_2^2
\end{equation}
and
\begin{equation}\label{e8.9}
\int_{\mathbb{R}^N}\nabla u_{+}(x-y_{\epsilon})\cdot \nabla
u_{\epsilon}(x)dx\leq \|u_{\epsilon}\|_2^2.
\end{equation}
Indeed, since $u_{+}$ is radial and  non-increasing, by  Lemma
\ref{lem jx}, we obtain that
\begin{equation*}
\begin{split}
\int_{\mathbb{R}^N}u_{+}(x-y)u_{\epsilon}(x)dx&\leq
\left(\frac{N}{|S^{N-1}|}\right)^{1/2}\sqrt{a}
\int_{\mathbb{R}^N}|x-y|^{-N/2}u_{\epsilon}(x)dx.
\end{split}
\end{equation*}
Noting that $supp(u_{\epsilon})\subset B_2$, and by using the
H\"{o}lder inequality,  we have,  for $|y|>10$,
\begin{equation*}
\begin{split}
\int_{\mathbb{R}^N}u_{+}(x-y)u_{\epsilon}(x)dx\lesssim
\int_{B_2}\left(\frac{|y|}{2}\right)^{-N/2}u_{\epsilon}(x)dx\lesssim
\left(\frac{|y|}{2}\right)^{-N/2}|B_2|^{1/2}\|u_{\epsilon}\|_2,
\end{split}
\end{equation*}
which combines with (\ref{e7.3}) imply that (\ref{e8.1}) holds for
$y_\epsilon$ large enough.

Noting that for any $y\in \mathbb{R}^N$, $u_+(x-y)$ is a solution to
the equation
\begin{equation*}
-\Delta u=\lambda u+(I_\alpha\ast|u|^{\bar{p}})|u|^{\bar{p}-2}u+\mu
|u|^{q-2}u
\end{equation*}
with some $\lambda<0$, we obtain that
\begin{equation}\label{e8.6}
\begin{split}
&\int_{\mathbb{R}^N}\nabla u_{+}(x-y)\cdot \nabla
u_{\epsilon}(x)dx\\
&\leq
\int_{\mathbb{R}^N}(I_\alpha\ast|u_{+}(x-y)|^{\bar{p}})|u_{+}(x-y)|^{\bar{p}-2}u_{+}(x-y)u_{\epsilon}(x)dx\\
&\qquad+\mu\int_{\mathbb{R}^N}|u_{+}(x-y)|^{q-2}u_{+}(x-y)u_{\epsilon}(x)dx.
\end{split}
\end{equation}
Similarly to the proof of (\ref{e8.1}), we have
\begin{equation}\label{e8.7}
\mu\int_{\mathbb{R}^N}|u_{+}(x-y)|^{q-2}u_{+}(x-y)u_{\epsilon}(x)dx\leq
\frac{1}{2}\|u_{\epsilon}\|_2^2
\end{equation}
for $y$ large enough. Now, for $x\in B_2$ and $|y|>100$, we
calculate
\begin{equation}\label{e8.2}
\begin{split}
&\int_{\mathbb{R}^N}\frac{1}{|x-z|^{N-\alpha}}|u_{+}(z-y)|^{\bar{p}}dz\\
&=\left(\int_{\mathbb{R}^N\backslash
B_{4|y|}(y)}+\int_{B_{4|y|}(y)\backslash
B_{\frac{|y|}{2}}(y)}+\int_{B_{\frac{|y|}{2}}(y)}\right)\frac{1}{|x-z|^{N-\alpha}}|u_{+}(z-y)|^{\bar{p}}dz\\
&:=I_1+I_2+I_3.
\end{split}
\end{equation}

It follows from $x\in B_2$, $|y|>100$ and $z\in
\mathbb{R}^N\backslash B_{4|y|}(y)$ that
\begin{equation*}
|x-z|\geq |z|-|x|\geq \frac{1}{2}|z|\geq \frac{1}{2}|y|
\end{equation*}
and
\begin{equation*}
|x-z|\geq \frac{1}{2}|z|\geq \frac{1}{4}(|z|+|y|)\geq
\frac{1}{4}|z-y|.
\end{equation*}
By using Lemma \ref{lem jx} with $t=2$, for any $\delta\in
(0,\min\{N-\alpha,N\bar{p}/2-\alpha\})$,  we have
\begin{equation}\label{e8.3}
\begin{split}
I_1&\lesssim\int_{\mathbb{R}^N\backslash
B_{4|y|}(y)}\frac{1}{|y|^{\delta}|z-y|^{N-\alpha-\delta}}|z-y|^{-\frac{N\bar{p}}{2}}\|u_+(z-y)\|_2^{\bar{p}}dz\\
&\lesssim\frac{1}{|y|^{\delta}}\int_{4|y|}^{+\infty}\frac{1}{r^{N-\alpha-\delta}}r^{-\frac{N\bar{p}}{2}}r^{N-1}dr
\lesssim \frac{1}{|y|^{\delta}}.
\end{split}
\end{equation}

For $x\in B_2$, $|y|>100$ and $z\in B_{4|y|}(y)\backslash
B_{\frac{|y|}{2}}(y)$, we obtain that $|z-x|\leq 8|y|$ and
$|z-y|\geq \frac{|y|}{2}$.  By using Lemma \ref{lem jx} with $t=2$,
we have
\begin{equation}\label{e8.4}
\begin{split}
I_2&\lesssim \int_{B_{4|y|}(y)\backslash
B_{\frac{|y|}{2}}(y)}\frac{1}{|x-z|^{N-\alpha}}\left(\frac{|y|}{2}\right)^{-\frac{N\bar{p}}{2}}\|u_+(z-y)\|_2^{\bar{p}}dz\\
&\lesssim
|y|^{-\frac{N\bar{p}}{2}}\int_{0}^{8|y|}\frac{1}{r^{N-\alpha}}r^{N-1}dr
\lesssim|y|^{\alpha-\frac{N\bar{p}}{2}}.
\end{split}
\end{equation}

For $x\in B_2$, $|y|>100$ and $z\in B_{\frac{|y|}{2}}(y)$, we obtain
that $|z-x|\geq \frac{|y|}{3}$ and  $|z-y|\leq \frac{|y|}{2}$. Then
by using Lemma \ref{lem jx} with $t=2^*$ and $\|u_+\|_{2^*}\leq C$,
we obtain that
\begin{equation}\label{e8.5}
\begin{split}
I_3&\lesssim
\int_{B_{\frac{|y|}{2}}(y)}\frac{1}{|x-z|^{N-\alpha}}|z-y|^{-N\bar{p}/{2^*}}dz\\
&\lesssim
\int_{B_{\frac{|y|}{2}}(0)}\frac{1}{\left(\frac{|y|}{3}\right)^{N-\alpha}}|z|^{-N\bar{p}/{2^*}}dz\\
&\lesssim
\frac{1}{|y|^{N-\alpha}}\int_{0}^{\frac{|y|}{2}}r^{-N\bar{p}/{2^*}}
r^{N-1}dr\lesssim |y|^{\frac{\alpha-N}{2}}.
\end{split}
\end{equation}

By using (\ref{e8.2}), (\ref{e8.3}), (\ref{e8.4}) and (\ref{e8.5}),
similarly to the proof of (\ref{e8.1}), we obtain
\begin{equation}\label{e8.8}
\begin{split}
\int_{\mathbb{R}^N}(I_\alpha\ast|u_{+}(x-y)|^{\bar{p}})|u_{+}(x-y)|^{\bar{p}-2}u_{+}(x-y)u_{\epsilon}(x)dx\leq
\frac{1}{2}\|u_{\epsilon}\|_2^2
\end{split}
\end{equation}
for $y$ large enough. In view of  (\ref{e8.6}), (\ref{e8.7}) and
(\ref{e8.8}), we complete the proof of (\ref{e8.9}).

\textbf{Step 2}. Let $y_\epsilon$ be given in Step 1 such that
(\ref{e8.1}) and (\ref{e8.9}) hold. As in  (\ref{e8.11}), we define
\begin{equation*}
\hat{u}_{\epsilon,t}=u_+(x-y_\epsilon)+tu_{\epsilon}(x)\
\mathrm{and}\
\bar{u}_{\epsilon,t}=\left(a^{-\frac{1}{2}}\|\hat{u}_{\epsilon,t}\|_2\right)^{\frac{N-2}{2}}
\hat{u}_{\epsilon,t}\left(a^{-\frac{1}{2}}\|\hat{u}_{\epsilon,t}\|_2x\right).
\end{equation*}
Similarly to the proof of  Lemma \ref{lem7.1}, there exists $t_0>0$
large enough and $\epsilon_0>0$ small enough such that
\begin{equation}\label{e8.12}
E(\bar{u}_{\epsilon,t})<m_a+\frac{2+\alpha}{2(N+\alpha)}S_\alpha^{\frac{N+\alpha}{2+\alpha}}
\end{equation}
for $t<\frac{1}{t_0}$ and $t>t_0$ uniformly for
$0<\epsilon<\epsilon_0$.

By using the inequality
\begin{equation*}
(a+b)^r\geq a^r+b^r,\ a>0,b>0, r\geq 1,
\end{equation*}
(\ref{e7.2}), (\ref{e7.3}), (\ref{e7.4}), (\ref{e7.5}), (\ref{e8.1})
and (\ref{e8.9}), similarly to (\ref{e9.1}), we obtain
\begin{equation}\label{e8.13}
\begin{split}
E(\bar{u}_{\epsilon,t})&\leq \frac{1}{2}\|\nabla
u_+\|_2^2+t\int_{\mathbb{R}^N}\nabla u_+(x-y_\epsilon)\nabla
u_{\epsilon}(x)+\frac{1}{2} \|\nabla (tu_{\epsilon})\|_2^2\\
&\quad
-\frac{1}{2\bar{p}}\int_{\mathbb{R}^N}(I_\alpha\ast|u_+|^{\bar{p}})|u_+|^{\bar{p}}
-\frac{1}{2\bar{p}}\int_{\mathbb{R}^N}(I_\alpha\ast|tu_{\epsilon}|^{\bar{p}})|tu_{\epsilon}|^{\bar{p}}\\
&\quad -\frac{\mu}{q}\|u_+\|_q^q-\frac{\mu}{q}\|tu_{\epsilon}\|_q^q\\
&\quad
-\frac{\mu(\gamma_q-1)}{2}\left(\frac{2t}{a}\int_{\mathbb{R}^N}
u_+(x-y_\epsilon)
u_{\epsilon}(x)+\frac{t^2}{a}\|u_{\epsilon}\|_2^2\right)\|\hat{u}_{\epsilon,t}\|_q^q\\
&\leq E(u_+)+E(tu_{\epsilon})+O(\|u_\epsilon\|_2^2)+O(\|u_\epsilon\|_2^2)\|\hat{u}_{\epsilon,t}\|_q^q\\
&<m_a+\frac{2+\alpha}{2(N+\alpha)}S_\alpha^{\frac{N+\alpha}{2+\alpha}}
\end{split}
\end{equation}
for $\frac{1}{t_0}<t<t_0$ uniformly for $\epsilon\in (0,\epsilon_0)$
small enough. In view of  (\ref{e8.12}) and (\ref{e8.13}), we
complete the proof.
\end{proof}

The next lemma is about the convergence of the Palais-Smale
sequence.

\begin{lemma}\label{pro1.4}
Assume  $N\geq 3$, $\alpha\in (0,N)$, $p=\bar{p}$, $q\in
(2,2+\frac{4}{N})$,  $\mu>0$, $a>0$ and $\mu
a^{\frac{q(1-\gamma_q)}{2}}\leq
(2K)^{\frac{q\gamma_q-2\bar{p}}{2(\bar{p}-1)}}$. Let $\{u_n\}\subset
S_{a,r}$ be a Palais-Smale sequence for $E|_{S_a}$ at level $c$,
with $P(u_n)\to 0$ as $n\to \infty$. If
\begin{equation*}
0<c<m_a+\frac{2+\alpha}{2(N+\alpha)}S_\alpha^{\frac{N+\alpha}{2+\alpha}},
\end{equation*}
then up a subsequence, $u_n\to u$ strongly in $H^1(\mathbb{R}^N)$,
and $u$ is a radial solution to (\ref{e1.5}) with $E(u)=c$ and some
$\lambda<0$.
\end{lemma}

\begin{proof}
The proof is divided into four steps.

\textbf{Step 1. We show  $\{u_n\}$ is bounded in
$H^1(\mathbb{R}^N)$.}  It follows from $P(u_n)=o_n(1)$ and
$E(u_n)=c+o_n(1)$ that
\begin{equation*}
E(u_n)=\left(\frac{1}{2}-\frac{1}{2\bar{p}}\right)\|\nabla
u_n\|_2^2+\left(\frac{\gamma_q}{2\bar{p}}-\frac{1}{q}\right)\mu\|
u_n\|_q^q+o_n(1).
\end{equation*}
Since $q\gamma_q<2<2\bar{p}$, by using the Gagliardo-Nirenberg
inequality, we obtain that
\begin{equation*}
\begin{split}
\left(\frac{1}{2}-\frac{1}{2\bar{p}}\right)\|\nabla u_n\|_2^2&\leq
c+\left(\frac{1}{q}-\frac{\gamma_q}{2\bar{p}}\right)\mu\|
u_n\|_q^q+o_n(1)\\
& \leq c+\left(\frac{1}{q}-\frac{\gamma_q}{2\bar{p}}\right)\mu
C_{N,q}^qa^{\frac{q}{2}(1-\gamma_q)}\|\nabla
u_n\|_2^{q\gamma_q}+o_n(1),
\end{split}
\end{equation*}
which implies that $\{\|\nabla u_n\|_2^2\}$ is bounded. Since
$\{u_n\}\subset S_a$, we obtain that $\{u_n\}$ is bounded in
$H^1(\mathbb{R}^N)$.

There exists $u\in H_{r}^1(\mathbb{R}^N)$ such that, up to a
subsequence, $u_n\rightharpoonup u$ weakly in $H^1(\mathbb{R}^N)$,
$u_n\to u$ strongly in $L^t(\mathbb{R}^N)$ with $t\in (2,2^*)$ and
$u_n\to u$ a.e. in $\mathbb{R}^N$.

\textbf{Step 2.  We claim that $u\not\equiv0$.} Suppose by
contradiction that $u\equiv 0$. By using $E(u_n)=c+o_n(1)$,
$P(u_n)=o_n(1)$, $\|u_n\|_q^{q}=o_n(1)$ and (\ref{e3.14}), we get
that
\begin{equation*}
E(u_n)=\left(\frac{1}{2}-\frac{1}{2\bar{p}}\right)\|\nabla
u_n\|_2^2+o_n(1)
\end{equation*}
and
\begin{equation}\label{e4.4}
\begin{split}
\|\nabla u_n\|_2^2&=\int_{\mathbb{R}^N}(I_\alpha\ast|u_n|^{\bar{p}})|u_n|^{\bar{p}}+o_n(1)\\
&\leq (S_\alpha^{-1}\|\nabla u_n\|_2^2)^{\bar{p}}+o_n(1).
\end{split}
\end{equation}
Since $c>0$, we obtain $\liminf_{n\to\infty}\|\nabla u_n\|_2^2>0$
and hence
$$\limsup_{n\to \infty}\|\nabla u_n\|_2^2\geq
S_\alpha^{\frac{N+\alpha}{2+\alpha}}.$$ Consequently,
\begin{equation*}
\begin{split}
c=\lim_{n\to\infty}\left\{\left(\frac{1}{2}-\frac{1}{2\bar{p}}\right)\|\nabla
u_n\|_2^2+o_n(1)\right\}\geq
\frac{2+\alpha}{2(N+\alpha)}S_\alpha^{\frac{N+\alpha}{2+\alpha}},
\end{split}
\end{equation*}
which contradicts to
\begin{equation*}
c<m_a+\frac{2+\alpha}{2(N+\alpha)}S_\alpha^{\frac{N+\alpha}{2+\alpha}}
\end{equation*}
and $m_a<0$. So $u\not\equiv0$.

\textbf{Step 3. We show $u$ is a solution to (\ref{e1.4}) with some
$\lambda<0$.} Since $\{u_n\}$ is a Palais-Smale sequence of
$E|_{S_a}$, by the Lagrange multipliers rule, there exists
$\lambda_n$ such that
\begin{equation}\label{e3.10}
\int_{\mathbb{R}^N}\left(\nabla u_n\cdot \nabla \varphi-\lambda_n
u_n\varphi-(I_\alpha\ast|u_n|^{\bar{p}})|u_n|^{\bar{p}-2}u_n\varphi-\mu|u_n|^{q-2}u_n\varphi\right)=o_n(1)\|\varphi\|_{H^1}
\end{equation}
for every $\varphi\in H^1(\mathbb{R}^N)$.  The choice $\varphi=u_n$
provides
\begin{equation}\label{e3.12}
\lambda_na=\|\nabla
u_n\|_2^2-\int_{\mathbb{R}^N}(I_\alpha\ast|u_n|^{\bar{p}})|u_n|^{\bar{p}}-\mu
\|u_n\|_q^{q}+o_n(1)
\end{equation}
and the boundedness of $\{u_n\}$ in $H^1(\mathbb{R}^N)$  implies
that $\lambda_n$ is bounded as well; thus, up to a subsequence
$\lambda_n\to \lambda\in \mathbb{R}$. Furthermore, by using
$P(u_n)=o_n(1)$, (\ref{e3.12}), $\mu>0$, $\gamma_q\in(0,1)$ and
$u_n\rightharpoonup u$ weakly in $H^1(\mathbb{R}^N)$, we obtain that
\begin{equation*}
\begin{split}
-\lambda_na&=\mu(1-\gamma_q) \|u_n\|_q^{q}+o_n(1)
\end{split}
\end{equation*}
and then
\begin{equation*}
-\lambda a\geq \mu(1-\gamma_q)\|u\|_q^{q}>0,
\end{equation*}
which implies that $\lambda<0$. By using (\ref{e3.10}) and Lemma
\ref{lem33.3}, we get that
\begin{equation}\label{e4.3}
\begin{split}
&\int_{\mathbb{R}^N}\left(\nabla u\cdot \nabla \varphi-\lambda
u\varphi-(I_\alpha\ast|u|^{\bar{p}})|u|^{\bar{p}-2}u\varphi
-\mu|u|^{q-2}u\varphi\right)\\
&=\lim_{n\to\infty}\int_{\mathbb{R}^N}\left(\nabla u_n\cdot \nabla
\varphi-\lambda_n
u_n\varphi-(I_\alpha\ast|u_n|^{\bar{p}})|u_n|^{\bar{p}-2}u_n\varphi
-\mu|u_n|^{q-2}u_n\varphi\right)\\
&=\lim_{n\to\infty}o_n(1)\|\varphi\|_{H^1}=0,
\end{split}
\end{equation}
which implies that $u$ satisfies the equation
\begin{equation}\label{e3.13}
-\Delta u=\lambda
u+(I_\alpha\ast|u|^{\bar{p}})|u|^{\bar{p}-2}u+\mu|u|^{q-2}u.
\end{equation}
Thus, $P(u)=0$ by Lemma \ref{lem3.8}.

\textbf{Step 4. We show $u_n\to u$ strongly in $H^1(\mathbb{R}^N)$.}
Set $v_n:=u_n-u$. Then we have
\begin{equation}\label{e1.31}
\|u_n\|_2^2=\|u\|_2^2+\|v_n\|_2^2+o_n(1),\ \|\nabla
u_n\|_2^2=\|\nabla u\|_2^2+\|\nabla v_n\|_2^2+o_n(1),
\end{equation}
\begin{equation}\label{e1.32}
\|u_n\|_q^q=\|u\|_q^q+\|v_n\|_q^q+o_n(1)=\|u\|_q^q+o_n(1)
\end{equation}
and
\begin{equation*}
\int_{\mathbb{R}^N}(I_\alpha\ast
|u_n|^{\bar{p}})|u_n|^{\bar{p}}=\int_{\mathbb{R}^N}(I_\alpha\ast
|u|^{\bar{p}})|u|^{\bar{p}}+\int_{\mathbb{R}^N}(I_\alpha\ast
|v_n|^{\bar{p}})|v_n|^{\bar{p}}+o_n(1),
\end{equation*}
which combines with $P(u_n)=o_n(1)$ and $P(u)=0$ give that
\begin{equation}\label{e4.5}
\|\nabla
v_n\|_2^2=\int_{\mathbb{R}^N}(I_\alpha\ast|v_n|^{\bar{p}})|v_n|^{\bar{p}}+o_n(1).
\end{equation}
Similarly to (\ref{e4.4}), we infer that
\begin{equation*}
\limsup_{n\to\infty}\|\nabla v_n\|_2^2\geq
S_\alpha^{\frac{N+\alpha}{2+\alpha}}\ \ \mathrm{or}\ \
\liminf_{n\to\infty}\|\nabla v_n\|_2^2=0.
\end{equation*}
If $\limsup_{n\to\infty}\|\nabla v_n\|_2^2\geq
S_\alpha^{\frac{N+\alpha}{2+\alpha}}$, then by using the fact that
$u$ satisfies (\ref{e3.13}), $\|u\|_2^2\leq a$, (\ref{e1.31}),
(\ref{e1.32}), and Lemma \ref{lem6.4}, we obtain that
\begin{equation*}
\begin{split}
E(u_n)&=\left(\frac{1}{2}-\frac{1}{2\bar{p}}\right)\|\nabla
u_n\|_2^2+\left(\frac{\gamma_q}{2\bar{p}}-\frac{1}{q}\right)\mu\|
u_n\|_q^q+o_n(1)\\
&=\left(\frac{1}{2}-\frac{1}{2\bar{p}}\right)\|\nabla
u\|_2^2+\left(\frac{\gamma_q}{2\bar{p}}-\frac{1}{q}\right)\mu\|
u\|^q+\left(\frac{1}{2}-\frac{1}{2\bar{p}}\right)\|\nabla
v_n\|_2^2+o_n(1)\\
&\geq m_{\|u\|_2^2}+\frac{2+\alpha}{2(N+\alpha)}S_\alpha^{\frac{N+\alpha}{2+\alpha}}+o_n(1)\\
&\geq
m_{a}+\frac{2+\alpha}{2(N+\alpha)}S_\alpha^{\frac{N+\alpha}{2+\alpha}}+o_n(1),
\end{split}
\end{equation*}
which contradicts $E(u_n)=c+o_n(1)$ and
$c<m_{a}+\frac{2+\alpha}{2(N+\alpha)}S_\alpha^{\frac{N+\alpha}{2+\alpha}}$.
Thus, $$\liminf_{n\to\infty}\|\nabla v_n\|_2^2=0$$ holds. So up to a
subsequence, $\nabla u_n\to \nabla u$ in $L^2(\mathbb{R}^N)$.
Choosing $\varphi=u_n-u$ in (\ref{e3.10}) and (\ref{e4.3}), and
subtracting, we obtain that
\begin{equation*}
\int_{\mathbb{R}^N}(|\nabla (u_n-u)|^2-\lambda|u_n-u|^2)\to 0.
\end{equation*}
Since $\lambda<0$, we get that $u_n\to u$ strongly in
$H^1(\mathbb{R}^N)$. The proof is complete.
\end{proof}

\textbf{Proof of Theorem \ref{thm6.3}}. It is a direct result of
Lemmas \ref{pro1.2}, \ref{pro1.3},  \ref{lem7.23}, \ref{lem7.21},
\ref{lem7.1}, \ref{lem7.3} and \ref{pro1.4}.

\section{Positivity, symmetry and exponential decay of solution $u$\\ to (\ref{e1.5}) with  $E(u)=\inf_{\mathcal{P}_{a,-}}E(v)$}
\setcounter{section}{5} \setcounter{equation}{0}

In this section, we prove Theorem \ref{thm1.3}. For future use, we
first give the following result.

\begin{lemma}\label{lem3.3}
Let $N\geq 3$, $\alpha\in (0,N)$, $p=\bar{p}$, $q\in
(2,2+\frac{4}{N})$,  $\mu>0$, $a>0$ and $\mu
a^{\frac{q(1-\gamma_q)}{2}}\leq
(2K)^{\frac{q\gamma_q-2\bar{p}}{2(\bar{p}-1)}}$.  If  $u\in
\mathcal{P}_{a,-}$ such that $E(u)=\inf_{\mathcal{P}_{a,-}}E(v)$,
then  $u$ satisfies the
 equation (\ref{e1.5}) with some $\lambda<0$.
\end{lemma}

\begin{proof}
By the Lagrange multipliers rule, there exist $\lambda$ and $\eta$
such that $u$ satisfies
\begin{equation}\label{e3.34}
\begin{split}
-\Delta u-&(I_\alpha\ast|u|^{\bar{p}})|u|^{\bar{p}-2}u-\mu|u|^{q-2}u\\
&\qquad=\lambda u+\eta[-2\Delta
u-2\bar{p}(I_\alpha\ast|u|^{\bar{p}})|u|^{\bar{p}-2}u-\mu
q\gamma_q|u|^{q-2}u],
\end{split}
\end{equation}
or equivalently,
\begin{equation*}
-(1-2\eta)\Delta u=\lambda u+(1-\eta
2\bar{p})(I_\alpha\ast|u|^{\bar{p}})|u|^{\bar{p}-2}u+\mu(1-\eta
q\gamma_q) |u|^{q-2}u.
\end{equation*}

 Next we show $\eta=0$. Similarly to the definition of $P(u)$ (see Lemma \ref{lem3.8}), we obtain
\begin{equation*}
(1-2\eta)\|\nabla u\|_2^2-(1-\eta 2\bar{p})
\int_{\mathbb{R}^N}(I_\alpha\ast|u|^{\bar{p}})|u|^{\bar{p}}-\mu(1-\eta
q\gamma_q)\gamma_q \|u\|_q^{q}=0,
\end{equation*}
which combines with $P(u)=0$ give that
\begin{equation*}
\eta\left(2\|\nabla
u\|_2^2-2\bar{p}\int_{\mathbb{R}^N}(I_\alpha\ast|u|^{\bar{p}})|u|^{\bar{p}}-\mu
q\gamma_q^2\|u\|_q^{q}\right)=0.
\end{equation*}
If $\eta\neq0$, then
\begin{equation*}
2\|\nabla
u\|_2^2-2\bar{p}\int_{\mathbb{R}^N}(I_\alpha\ast|u|^{\bar{p}})|u|^{\bar{p}}-\mu
q\gamma_q^2\|u\|_q^{q}=0,
\end{equation*}
which combines with $P(u)=0$ give that
\begin{equation*}
\begin{cases}
\mu\gamma_q(2\bar{p}-q\gamma_q)\|u\|_q^q=(2\bar{p}-2)\|\nabla u\|_2^2,   \\
(q\gamma_q-2\bar{p})\int_{\mathbb{R}^N}(I_\alpha\ast|u|^{\bar{p}})|u|^{\bar{p}}=(q\gamma_q-2)\|\nabla
u\|_2^2.
\end{cases}
\end{equation*}
Hence,
\begin{equation*}
E(u)=\frac{(\bar{p}-1)(q\gamma_q-2)}{2\bar{p}q\gamma_q}\|\nabla
u\|_2^2<0,
\end{equation*}
which contradicts to $\inf_{\mathcal{P}_{a,-}}E(v)\geq 0$, see Lemma
\ref{lem6.6}. So $\eta=0$.

It follows from (\ref{e3.34}) with $\eta=0$, $P(u)=0$, $0<
\gamma_q<1$ and $\mu>0$ that
\begin{equation*}
\begin{split}
\lambda a=\|\nabla u\|_2^2-
\int_{\mathbb{R}^N}(I_\alpha\ast|u|^{\bar{p}})|u|^{\bar{p}}-\mu\|u\|_q^{q}=\mu
(\gamma_q-1) \|u\|_q^{q}<0,
\end{split}
\end{equation*}
which implies $\lambda<0$. The proof is complete.
\end{proof}

Nextly, we study the positivity of the solution $u$ to (\ref{e1.5})
with $E(u)=\inf_{\mathcal{P}_{a,-}}E(v)$. By Lemma \ref{lem3.3}, it
is enough to prove the following result.

\begin{proposition}\label{lem4.5}
Let $N\geq 3$, $\alpha\in (0,N)$, $p=\bar{p}$, $q\in
(2,2+\frac{4}{N})$,  $\mu>0$, $a>0$ and $\mu
a^{\frac{q(1-\gamma_q)}{2}}\leq
(2K)^{\frac{q\gamma_q-2\bar{p}}{2(\bar{p}-1)}}$.  If $u \in
\mathcal{P}_{a,-}$  such that $E(u)=\inf_{\mathcal{P}_{a,-}}E(v)$,
then $|u|_{\tau_{|u|}^-}\in \mathcal{P}_{a,-}$ and
$E(|u|_{\tau_{|u|}^-})=\inf_{\mathcal{P}_{a,-}}E(v)$. Moreover,
$|u|_{\tau_{|u|}^-}>0$ in $\mathbb{R}^N$.
\end{proposition}

\begin{proof}
It follows from $\|\nabla |u|\|_2^2\leq \|\nabla u\|_2^2$ that
$\Psi_{|u|}(\tau)\leq \Psi_{u}(\tau)$ for any $\tau>0$. By Lemma
\ref{lem6.6}, we have
\begin{equation*}
E(|u|_{\tau_{|u|}^-})=\Psi_{|u|}(\tau_{|u|}^-)\leq
\Psi_{u}(\tau_{|u|}^-)\leq \Psi_{u}(\tau_{u}^-)=E(u).
\end{equation*}
Since $|u|_{\tau_{|u|}^-}\in \mathcal{P}_{a,-}$, we obtain that
$E(|u|_{\tau_{|u|}^-})=\inf_{\mathcal{P}_{a,-}}E(v)$. By Lemma
\ref{lem3.3}, there exists $\lambda<0$ such that
$|u|_{\tau_{|u|}^-}$ satisfies the equation
\begin{equation*}
-\Delta u=\lambda u+(I_\alpha\ast
|u|^{\bar{p}})|u|^{\bar{p}-2}u+\mu|u|^{q-2}u.
\end{equation*}
Since $|u|_{\tau_{|u|}^-}$ is continuous by Theorem 2.1 in
\cite{Li-Ma 2020}, the strong maximum principle implies  that
$|u|_{\tau_{|u|}^-}>0$ in $\mathbb{R}^N$.
\end{proof}

Nextly, we study the radial symmetry of the solution $u$ to
(\ref{e1.5}) with $E(u)=\inf_{\mathcal{P}_{a,-}}E(v)$. We follow the
arguments of \cite{Moroz-Schaftingen 2015} which rely on
polarization. So we first recall some theories of polarization
(\cite{{Brock-Solynin 2000},{Moroz-Schaftingen JFA 2013},{Van
Schaftingen-Willem 2008}}).

Assume that $H\subset \mathbb{R}^N$ is a closed half-space and that
$\sigma_H$ is the reflection with respect to $\partial H$. The
polarization $u^H : \mathbb{R}^N \to \mathbb{R}$ of $u :
\mathbb{R}^N \to \mathbb{R}$ is defined for $x\in \mathbb{R}^N$ by
\begin{equation*}
u^H(x)=\left\{
\begin{array}{ll}
\max\{u(x),u(\sigma_H(x))\}, &\ \mathrm{if}\ x\in H,\\
\min\{u(x),u(\sigma_H(x))\}, &\ \mathrm{if}\ x\not\in H.
\end{array}
\right.
\end{equation*}

\begin{lemma}\label{lem4.1}
(Polarization and Dirichlet integrals, Lemma 5.3 in
\cite{Brock-Solynin 2000}). Let $H\subset \mathbb{R}^N$ be a closed
half-space. If $u\in H^1(\mathbb{R}^N)$, then $u^H\in
H^1(\mathbb{R}^N)$ and
\begin{equation*}
\int_{\mathbb{R}^N}|\nabla u^H|^2=\int_{\mathbb{R}^N}|\nabla u|^2.
\end{equation*}
\end{lemma}

\begin{lemma}\label{lem4.2}
(Polarization and nonlocal integrals, Lemma 5.3 in
\cite{Moroz-Schaftingen JFA 2013}). Let $\alpha\in(0,N)$, $u\in
L^{\frac{2N}{N+\alpha}} (\mathbb{R}^N)$ and $H\subset \mathbb{R}^N$
be a closed half-space. If $u\geq 0$, then
\begin{equation*}
\int_{\mathbb{R}^N}\int_{\mathbb{R}^N}\frac{u(x)u(y)}{|x-y|^{N-\alpha}}dxdy
\leq
\int_{\mathbb{R}^N}\int_{\mathbb{R}^N}\frac{u^H(x)u^H(y)}{|x-y|^{N-\alpha}}dxdy,
\end{equation*}
with equality if and only if either $u^H = u$ or $u^H = u\circ
\sigma_H$.
\end{lemma}

\begin{lemma}\label{lem4.4}
(Symmetry and polarization, Proposition 3.15 in \cite{Van
Schaftingen-Willem 2008}, Lemma 5.4 in \cite{Moroz-Schaftingen JFA
2013}). Assume that $u\in L^2(\mathbb{R}^N)$ is nonnegative. There
exist $x_0\in \mathbb{R}^N$ and a non-increasing function $v :
(0,\infty) \to \mathbb{R}$ such that for almost every $x\in
\mathbb{R}^N$, $u(x)=v(|x-x_0|)$ if and only if for every closed
half-space $H\subset \mathbb{R}^N$, $u^H = u$ or $u^H = u\circ
\sigma_H$.
\end{lemma}

Now we are ready to prove the radial symmetry  result.

\begin{proposition}\label{lem4.6}
Let $N\geq 3$, $\alpha\in (0,N)$, $p=\bar{p}$, $q\in
(2,2+\frac{4}{N})$,  $\mu>0$, $a>0$ and $\mu
a^{\frac{q(1-\gamma_q)}{2}}\leq
(2K)^{\frac{q\gamma_q-2\bar{p}}{2(\bar{p}-1)}}$.  If $u$ is a
positive solution to (\ref{e1.5}) with
$E(u)=\inf_{\mathcal{P}_{a,-}}E(v)$, then there exist $x_0\in
\mathbb{R}^N$ and a non-increasing positive function $v : (0,\infty)
\to \mathbb{R}$ such that $u(x)=v(|x-x_0|)$ for almost every $x\in
\mathbb{R}^N$.
\end{proposition}

\begin{proof}
By Lemmas \ref{lem3.8} and \ref{lem6.6}, $u\in \mathcal{P}_{a,-}$.
Let $\Gamma(a)$
 be defined in (\ref{e6.8}), $\tau_1\geq 0$ be small enough  such that $u_{\tau_1}\in
 V_a$ and $E(u_{\tau_1})<0$. Then $g_u(t)=u_{t+\tau_1}\in \Gamma(a)$,
 $g_u(\tau_{u}^--\tau_1)=u_{\tau_{u}^-}$, $g_u(t)\geq 0$ for every
 $t\geq 0$, $E(g_u(t))<E(u_{\tau_{u}^-})=E(u)=\inf_{v\in\mathcal{P}_{a,-}}E(v)$
 for any
$t\in \left([0,\infty)\setminus \{\tau_{u}^--\tau_1\}\right)$.

For every closed half-space $H$ define the path $g_u^H : [0, \infty)
\to  S_a$ by $g_u^H(t) = (g_u(t))^H$. By Lemma \ref{lem4.1} and
$\|u^H\|_r=\|u\|_r$ with $r\in [1,\infty)$, we have $g_u^H \in C([0,
\infty),S_a)$. By Lemmas \ref{lem4.1} and \ref{lem4.2}, we obtain
that  $g_u^H(0)\in V_a$ and $E(g_u^H(t))\leq E(g_u(t))$ for every
$t\in [0,\infty)$ and thus $g_u^H \in \Gamma(a)$. Hence,
\begin{equation*}
\max_{t\in[0,\infty)}E(g_u^H(t))\geq
\inf_{v\in\mathcal{P}_{a,-}}E(v).
\end{equation*}
Since for every $t\in \left([0,\infty)\setminus
\{\tau_{u}^--\tau_1\}\right)$,
\begin{equation*}
E(g_u^H(t))\leq E(g_u(t))<E(u)=\inf_{v\in\mathcal{P}_{a,-}}E(v),
\end{equation*}
we deduce that
\begin{equation*}
E\left(g_u^H\left( \tau_{u}^--\tau_1 \right)\right)
=E(u^{H})=\inf_{v\in\mathcal{P}_{a,-}}E(v).
\end{equation*}
Hence $E(u^H)=E(u)$, which implies that
\begin{equation*}
\int_{\mathbb{R}^N}\left(I_\alpha\ast
\left|u^H\right|^{\bar{p}}\right)\left|u^H\right|^{\bar{p}}=\int_{\mathbb{R}^N}(I_\alpha\ast
|u|^{\bar{p}})|u|^{\bar{p}}.
\end{equation*}
By Lemma \ref{lem4.2}, we have $u^H=u$ or $u^H=u\circ \sigma_H$. By
Lemma \ref{lem4.4}, we complete the proof.
\end{proof}

\begin{proposition}\label{pro1.5}
Let $N\geq 3$, $\alpha\in (0,N)$, $p=\bar{p}$,
$q\in(2,2+\frac{4}{N})$,  $\mu>0$, $a>0$, $\mu
a^{\frac{q(1-\gamma_q)}{2}}\leq
(2K)^{\frac{q\gamma_q-2\bar{p}}{2(\bar{p}-1)}}$ and $\alpha\geq N-4$
(i.e., $\bar{p}\geq 2$). If $u$ is a positive solution to
(\ref{e1.5}) with $E(u)=\inf_{\mathcal{P}_{a,-}}E(v)$, then $u$ has
exponential decay at infinity:
\begin{equation*}
u(x)\leq C e^{-\delta |x|},\ \ |x|\geq r_0,
\end{equation*}
for some $C>0$, $\delta>0$ and  $r_0>0$.
\end{proposition}

\begin{proof}
By Lemmas \ref{lem3.8} and \ref{lem3.3}, there exists $\lambda<0$
such that $u$ satisfies the equation
\begin{equation}\label{e8.14}
-\Delta u=\lambda
u+(I_\alpha\ast|u|^{\bar{p}})|u|^{\bar{p}-2}u+\mu|u|^{q-2}u.
\end{equation}
By Proposition \ref{lem4.6}, there exist $x_0\in \mathbb{R}^N$ and a
non-increasing positive function $v : (0,\infty) \to \mathbb{R}$
such that $u(x)=v(|x-x_0|)$ for almost every $x\in \mathbb{R}^N$.
Hence, $w:=u(x+x_0)$ is a positive and radially non-increasing
solution to (\ref{e8.14}). Similarly to the estimate of
(\ref{e8.2}), there exists $r_0>0$ such that
\begin{equation*}
(I_\alpha\ast|w|^{\bar{p}})(x)=C\int_{\mathbb{R}^N}\frac{|w(x-z)|^{\bar{p}}}{|z|^{N-\alpha}}dz\leq
-\frac{\lambda}{2}
\end{equation*}
for $|x|>r_0$. Hence, if $\bar{p}>2$, there exists $C>0$ such that
$w$ satisfies
\begin{equation*}
-\Delta w\leq \lambda w +C w^{\bar{p}-2}w+\mu  w^{q-2}w,\ \ |x|\geq
r_0,
\end{equation*}
and  if $\bar{p}=2$, $w$ satisfies
\begin{equation*}
-\Delta w\leq \frac{\lambda}{2} w +\mu  w^{q-2}w,\ \ \ \ |x|\geq
r_0.
\end{equation*}
Now, repeating word by word the proof of Lemma 2 in
\cite{Berestycki-Lions 1983}, we can show that $w$ decays
exponentially at infinity. The proof is complete.
\end{proof}

\textbf{Proof of Theorem \ref{thm1.3}}. By Proposition \ref{lem4.5},
$w:=|u|_{\tau_{|u|}^-}\in \mathcal{P}_{a,-}$ is a positive solution
to (\ref{e1.5}) with $E(w)=\inf_{v\in \mathcal{P}_{a,-}}E(v)$.
Hence, $w$ has exponential decay at infinity by Proposition
\ref{pro1.5}, and by Proposition  \ref{lem4.6}, there exist $x_0\in
\mathbb{R}^N$ and a non-increasing positive function $v : (0,\infty)
\to \mathbb{R}$ such that $w=v(|x-x_0|)$ for almost every $x\in
\mathbb{R}^N$. The proof is complete by using the fact
$|u|(x)=(\tau_{|u|}^-)^{-N/2}w\left(\frac{x}{\tau_{|u|}^-}\right)$.

\section{Dynamical studies to the equation (\ref{e1.2})}

\setcounter{section}{6} \setcounter{equation}{0}

In this section, we first study the local existence, global
existence, and the finite time blow up to the Cauchy problem
(\ref{e1.2}), and then study the stability and instability of the
standing waves obtained in Sections 3 and 4.

\subsection{Local existence}

In this subsection, we consider the local existence to the Cauchy
problem
\begin{equation}\label{e5.1}
\begin{cases}
i\partial_t\varphi+\Delta
\varphi+(I_{\alpha}\ast|\varphi|^{\bar{p}})|\varphi|^{\bar{p}-2}\varphi+\mu|\varphi|^{q-2}\varphi,
\quad (t,x)\in
\mathbb{R}\times\mathbb{R}^{N},\\
\varphi(0,x)=\varphi_0(x)\in H^1(\mathbb{R}^N),\quad x\in
\mathbb{R}^{N}.
\end{cases}
\end{equation}

\begin{definition}\label{def5.1}
Let $N\geq 3$. The pair $(p, r)$ is said to be Schr\"{o}dinger
admissible, for short $(p, r)\in S$, if
\begin{equation*}
\frac{2}{p}+\frac{N}{r}=\frac{N}{2},\ p, r\in [2,\infty].
\end{equation*}
\end{definition}

Define
\begin{equation}\label{e5.2}
(p_1,r_1):=\left(\frac{2(N+\alpha)}{N-2}=2\bar{p},\frac{2N(N+\alpha)}{N\alpha+4+N^2-2N}\right)
\end{equation}
and
\begin{equation}\label{e5.3}
(p_2,r_2):=\left(\frac{4q}{(q-2)(N-2)},\frac{Nq}{q+N-2}\right).
\end{equation}
Then  $(p_1,r_1)$, $(p_2,r_2)\in S$ by direct calculation. For such
defined admissible pairs, we define the spaces
$Y_T:=Y_{p_1,r_1,T}\cap Y_{p_2,r_2,T}$ and $X_T:=X_{p_1,r_1,T}\cap
X_{p_2,r_2,T}$ equipped with the following norms:
\begin{equation}\label{e5.4}
\|\psi\|_{Y_T}=\|\psi\|_{Y_{p_1,r_1,T}}+\|\psi\|_{Y_{p_2,r_2,T}}\
\mathrm{and}\
\|\psi\|_{X_T}=\|\psi\|_{X_{p_1,r_1,T}}+\|\psi\|_{X_{p_2,r_2,T}},
\end{equation}
where, for any $p,\, r\in (1,\infty)$,
\begin{equation*}
\|\psi(t,x)\|_{Y_{p,r,T}}:=\left(\int_{0}^{T}\|\psi(t,\cdot)\|_r^pdt\right)^{1/p}
\end{equation*}
and
\begin{equation*}
\|\psi(t,x)\|_{X_{p,r,T}}:=\left(\int_{0}^{T}\|\psi(t,\cdot)\|_{W^{1,r}}^pdt\right)^{1/p}.
\end{equation*}

\begin{definition}\label{def5.2}
Let $T>0$. We say that $\varphi(t,x)$ is an integral solution of the
Cauchy problem (\ref{e5.1}) on the time interval $[0,T]$ if
$\varphi\in C([0,T],H^1(\mathbb{R}^N))\cap X_{T}$ and
$\varphi(t,x)=e^{it\Delta}\varphi_0(x)-i\int_{0}^{t}e^{i(t-s)\Delta}g(\varphi(s,x))ds$
for all $t\in(0,T)$, where $g(\varphi):=g_1(\varphi)+g_2(\varphi)$,
$g_1(\varphi):=(I_{\alpha}\ast|\varphi|^{\bar{p}})|\varphi|^{\bar{p}-2}\varphi$
and $g_2(\varphi):=\mu|\varphi|^{q-2}\varphi$.
\end{definition}

Let us recall Strichartz's estimates that will be useful in the
sequel (see for example (\cite{Cazenave 2003}, Theorem 2.3.3 and
Remark 2.3.8) and \cite{Keel-Tao 1998} for the endpoint estimates).

\begin{lemma}\label{lem5.1}
Let $N\geq 3$, $(p,r)$ and $(\tilde{p},\tilde{r})\in S$. Then there
exists a constant $C>0$ such that for any $T>0$, the following
properties hold:

(i) For any $u\in L^2(\mathbb{R}^N)$, the function $t\mapsto
e^{it\Delta}u$ belongs to $Y_{p,r,T}\cap C([0,T],L^2(\mathbb{R}^N))$
and $\|e^{it\Delta}u\|_{Y_{p,r,T}}\leq C\|u\|_2$.

(ii) Let $F(t,x)\in Y_{\tilde{p}',\tilde{r}',T}$, where we use a
prime to denote conjugate indices. Then the function
\begin{equation*}
t\mapsto \Phi_F(t,x):=\int_{0}^te^{i(t-s)\Delta}F(s,x)ds
\end{equation*}
belongs to $Y_{p,r,T}\cap C([0,T ],L^2(\mathbb{R}^N))$ and
$\|\Phi_F\|_{Y_{p,r,T}}\leq C\|F\|_{Y_{\tilde{p}',\tilde{r}',T}}$.

(iii) For every $u \in H^1(\mathbb{R}^N)$, the function $t\mapsto
e^{it\Delta}u$ belongs to $X_{p,r,T} \cap
C([0,T],H^1(\mathbb{R}^N))$ and $\|e^{it\Delta}u\|_{X_{p,r,T}}\leq
C\|u\|_{H^1}$.
\end{lemma}

\begin{lemma}\label{lem5.2}
Let $N\geq 3$, $\alpha\in (0,N)$, $p=\bar{p}$,  $\alpha\geq N-4$
(i.e., $\bar{p}\geq 2$), $\alpha< N-2$, $(p_1,r_1)$ be defined in
(\ref{e5.2}) and $g_1(\varphi)$ be defined in Definition
\ref{def5.2}. Then for every $(\tilde{p},\tilde{r})\in S$ there
exists a constant $C > 0$ such that for every $T > 0$,
\begin{equation}\label{e5.10}
\left\|\int_{0}^{t}e^{i(t-s)\Delta}[\nabla
g_1(\varphi(s))]ds\right\|_{Y_{\tilde{p},\tilde{r},T}}\leq C\|\nabla
\varphi\|_{Y_{p_1,r_1,T}}^{2\bar{p}-1}
\end{equation}
and
\begin{equation}\label{e5.11}
\begin{split}
&\left\|\int_{0}^{t}e^{i(t-s)\Delta}[g_1(\varphi(s))-g_1(\psi(s))]ds\right\|_{Y_{\tilde{p},\tilde{r},T}}\\
&\qquad \qquad \leq C(\|\nabla
\varphi\|_{Y_{p_1,r_1,T}}^{2\bar{p}-2}+\|\nabla
\psi\|_{Y_{p_1,r_1,T}}^{2\bar{p}-2})\|\varphi-\psi\|_{Y_{p_1,r_1,T}}.
\end{split}
\end{equation}
\end{lemma}

\begin{proof}
By using
\begin{equation*}
|\nabla (|\varphi|^{\bar{p}})|\lesssim |\varphi|^{\bar{p}-1}|\nabla
\varphi|\ \ \mathrm{and}\ \
|\nabla(|\varphi|^{\bar{p}-2}\varphi)|\lesssim
|\varphi|^{\bar{p}-2}|\nabla \varphi|,
\end{equation*}
we obtain that
\begin{equation}\label{e5.9}
\begin{split}
|\nabla g_1(\varphi)|&\lesssim
\left|(I_\alpha\ast|\varphi|^{\bar{p}})|\varphi|^{\bar{p}-2}\nabla\varphi\right|
+\left|(I_\alpha\ast[|\varphi|^{\bar{p}-1}\nabla
\varphi])|\varphi|^{\bar{p}-2}\varphi\right|\\
&:=I_1+I_2.
\end{split}
\end{equation}
By using
\begin{equation*}
\left||\varphi|^{\bar{p}}-|\psi|^{\bar{p}}\right|\lesssim\left(|\varphi|+|\psi|\right)^{\bar{p}-1}|\varphi-\psi|
\end{equation*}
and
\begin{equation*}
\left||\varphi|^{\bar{p}-2}\varphi-|\psi|^{\bar{p}-2}\psi\right|\lesssim\left(|\varphi|^{\bar{p}-2}+|\psi|^{\bar{p}-2}\right)|\varphi-\psi|,
\end{equation*}
we obtain that
\begin{equation*}
\begin{split}
|g_1(\varphi)-g_1(\psi)|
&\leq \left|(I_\alpha\ast|\varphi|^{\bar{p}})(|\varphi|^{\bar{p}-2}\varphi-|\psi|^{\bar{p}-2}\psi)\right|\\
&\qquad\qquad \qquad +\left|(I_\alpha\ast|\varphi|^{\bar{p}}-I_\alpha\ast|\psi|^{\bar{p}})|\psi|^{\bar{p}-2}\psi\right|\\
&\lesssim
\left|(I_\alpha\ast|\varphi|^{\bar{p}})[\left(|\varphi|^{\bar{p}-2}+|\psi|^{\bar{p}-2}\right)|\varphi-\psi|]\right|\\
&\qquad\qquad \qquad +\left|(I_\alpha\ast
[\left(|\varphi|+|\psi|\right)^{\bar{p}-1}|\varphi-\psi|])|\psi|^{\bar{p}-2}\psi\right|\\
&:=I_3+I_4.
\end{split}
\end{equation*}

Case $\bar{p}>2$. Set
\begin{equation*}
a_1=\frac{2N}{N-2-\alpha}\ \ \mathrm{and}\ \
q_1=\frac{2N(N+\alpha)}{(\alpha+4-N)(N-2+\alpha)},
\end{equation*}
then
\begin{equation*}
\frac{Na_1}{N+\alpha a_1}\in \left(1,\frac{N}{\alpha}\right),
\frac{Na_1}{N+\alpha
a_1}\bar{p}=(\bar{p}-2)q_1=r_1^*:=\frac{Nr_1}{N-r_1},
\frac{1}{r_1'}=\frac{1}{a_1}+\frac{1}{q_1}+\frac{1}{r_1}.
\end{equation*}
By using the  H\"{o}lder inequality, the Hardy-Littlewood-Sobolev
inequality and the Sobolev embedding
$W^{1,r_1}(\mathbb{R}^N)\hookrightarrow L^{r_1^*}(\mathbb{R}^N)$, we
have
\begin{equation}\label{e5.7}
\begin{split}
\|I_1\|_{r_1'}
&\leq\|I_\alpha\ast|\varphi|^{\bar{p}}\|_{a_1}\||\varphi|^{\bar{p}-2}\|_{q_1}\|\nabla \varphi\|_{r_1}\\
&\lesssim \||\varphi|^{\bar{p}}\|_{\frac{Na_1}{N+\alpha a_1}}\|\varphi\|_{(\bar{p}-2)q_1}^{\bar{p}-2}\|\nabla \varphi\|_{r_1}\\
&=\|\varphi\|_{\frac{Na_1}{N+\alpha a_1}\bar{p}}^{\bar{p}}\|\varphi\|_{(\bar{p}-2)q_1}^{\bar{p}-2}\|\nabla \varphi\|_{r_1}\\
&\lesssim \|\nabla \varphi\|_{r_1}^{2\bar{p}-1}.
\end{split}
\end{equation}
Set
\begin{equation*}
a_2=\frac{2N}{N-\alpha}\ \ \mathrm{and}\ \
q_2=\frac{2N(N+\alpha)}{(\alpha+2)(N-2+\alpha)},
\end{equation*}
then
\begin{equation*}
\frac{1}{r_1'}=\frac{1}{a_2}+\frac{1}{q_2},\ \frac{Na_2}{N+\alpha
a_2}\in \left(1,\frac{N}{\alpha}\right),\
\frac{1}{\frac{Na_2}{N+\alpha a_2}}=\frac{1}{r_1}+\frac{1}{q_2}\ \
\mathrm{and}\ \ (\bar{p}-1)q_2=r_1^*.
\end{equation*}
By using the  H\"{o}lder inequality, the Hardy-Littlewood-Sobolev
inequality and the Sobolev embedding
$W^{1,r_1}(\mathbb{R}^N)\hookrightarrow L^{r_1^*}(\mathbb{R}^N)$, we
have
\begin{equation}\label{e5.8}
\begin{split}
\|I_2\|_{r_1'}
&\leq\|I_\alpha\ast(|\varphi|^{\bar{p}-1}\nabla \varphi)\|_{a_2}\||\varphi|^{\bar{p}-2}\varphi\|_{q_2}\\
&\lesssim \||\varphi|^{\bar{p}-1}\nabla \varphi\|_{\frac{Na_2}{N+\alpha a_2}}\|\varphi\|_{(\bar{p}-1)q_2}^{\bar{p}-1}\\
&\lesssim \||\varphi|^{\bar{p}-1}\|_{q_2}\|\nabla
\varphi\|_{r_1}\|\varphi\|_{(\bar{p}-1)q_2}^{\bar{p}-1}\\
&=\|\varphi\|_{(\bar{p}-1)q_2}^{\bar{p}-1}\|\varphi\|_{(\bar{p}-1)q_2}^{\bar{p}-1}\|\nabla
\varphi\|_{r_1}\\
&\lesssim \|\nabla \varphi\|_{r_1}^{2\bar{p}-1}.
\end{split}
\end{equation}
By using (\ref{e5.9}), (\ref{e5.7}) and  (\ref{e5.8}), we have
\begin{equation*}
\begin{split}
\|\nabla
g_1(\varphi)\|_{Y_{p_1',r_1',T}}&=\left(\int_{0}^{T}\|\nabla
g_1(\varphi(t)) \|_{r_1'}^{p_1'}dt\right)^{\frac{1}{p_1'}}\\
&\lesssim \left(\int_{0}^{T}\|\nabla
\varphi(t) \|_{r_1}^{(2\bar{p}-1)p_1'}dt\right)^{\frac{1}{p_1'}}\\
&=\|\nabla \varphi\|_{Y_{p_1,r_1,T}}^{2\bar{p}-1}.
\end{split}
\end{equation*}
Hence, by Lemma \ref{lem5.1} (ii), we obtain that
\begin{equation*}
\left\|\int_{0}^{t}e^{i(t-s)\Delta}[\nabla
g_1(\varphi(s))]ds\right\|_{Y_{\tilde{p},\tilde{r},T}}\lesssim\|\nabla
g_1(\varphi)\|_{Y_{p_1',r_1',T}}\lesssim\|\nabla
\varphi\|_{Y_{p_1,r_1,T}}^{2\bar{p}-1};
\end{equation*}
that is, (\ref{e5.10}) holds.

Similarly to (\ref{e5.7}) and (\ref{e5.8}), we obtain that
\begin{equation}\label{e5.12}
\|I_3\|_{r_1'}\lesssim
\left(\|\varphi\|_{r_1^*}+\|\psi\|_{r_1^*}\right)^{2\bar{p}-2}\|\varphi-\psi\|_{r_1}
\end{equation}
and
\begin{equation}\label{e5.13}
\|I_4\|_{r_1'}\lesssim
\left(\|\varphi\|_{r_1^*}+\|\psi\|_{r_1^*}\right)^{2\bar{p}-2}\|\varphi-\psi\|_{r_1}.
\end{equation}
By using the H\"{o}lder inequality, we have
\begin{equation*}
\begin{split}
&\|I_3\|_{Y_{p_1',r_1',T}}\\
&\lesssim
\left(\int_{0}^{T}(\|\varphi(t)\|_{r_1^*}
+\|\psi(t)\|_{r_1^*})^{(2\bar{p}-2)p_1'}\|\varphi(t)-\psi(t)\|_{r_1}^{p_1'}dt\right)^{\frac{1}{p_1'}}\\
& \lesssim
\left(\int_0^T(\|\varphi(t)\|_{r_1^*}+\|\psi(t)\|_{r_1^*})^{p_1}dt\right)^{\frac{(2\bar{p}-2)p_1'}{p_1}\frac{1}{p_1'}}
\left(\int_{0}^T\|\varphi(t)-\psi(t)\|_{r_1}^{p_1}dt\right)^{\frac{p_1'}{p_1}\frac{1}{p_1'}}\\
&\lesssim (\|\nabla \varphi\|_{Y_{p_1,r_1,T}}+\|\nabla
\psi\|_{Y_{p_1,r_1,T}})^{2\bar{p}-2}\|\varphi(t)-\psi(t)\|_{Y_{p_1,r_1,T}}.
\end{split}
\end{equation*}
Hence, by Lemma \ref{lem5.1} (ii), we obtain  (\ref{e5.11}) holds.\\

Case $\bar{p}=2$. Similarly to case $\bar{p}>2$, just in the
estimate of $\|I_1\|_{r_1'}$ and $\|I_3\|_{r_1'}$ by choosing
$q_1=\infty$ and $a_1=\frac{2N}{N-2-\alpha}$, we have (\ref{e5.7}),
(\ref{e5.8}), (\ref{e5.12}) and (\ref{e5.13}) hold and then
(\ref{e5.10}) and (\ref{e5.11}) hold. The proof is complete.
\end{proof}

The following lemma is cited from \cite{JEANJEAN-JENDREJ}.

\begin{lemma}\label{lem5.3}
Let $N\geq 3$, $q\in(2,2^*)$,  $(p_2,r_2)$ be defined in
(\ref{e5.3}) and $g_2(\varphi)$ be defined in Definition
\ref{def5.2}. Then for every $(\tilde{p},\tilde{r})\in S$ there
exists a constant $C > 0$ such that for every $T > 0$,
\begin{equation*}
\left\|\int_{0}^{t}e^{i(t-s)\Delta}[\nabla
g_2(\varphi(s))]ds\right\|_{Y_{\tilde{p},\tilde{r},T}}\leq
CT^{\frac{(N-2)(2^*-q)}{4}}\|\nabla \varphi\|_{Y_{p_2,r_2,T}}^{q-1}
\end{equation*}
and
\begin{equation*}
\begin{split}
&\left\|\int_{0}^{t}e^{i(t-s)\Delta}[g_2(\varphi(s))-g_2(\psi(s))]ds\right\|_{Y_{\tilde{p},\tilde{r},T}}\\
&\qquad \qquad \leq CT^{\frac{(N-2)(2^*-q)}{4}}(\|\nabla
\varphi\|_{Y_{p_2,r_2,T}}^{q-2}+\|\nabla
\psi\|_{Y_{p_2,r_2,T}}^{q-2})\|\varphi-\psi\|_{Y_{p_2,r_2,T}}.
\end{split}
\end{equation*}
\end{lemma}

Similarly to the proof of Lemma 3.7 in \cite{JEANJEAN-JENDREJ}, we
have the following result.
\begin{lemma}\label{lem5.4}
For all $R,T > 0$ the metric space $(B_{R,T} ,d)$ is complete, where
\begin{equation*}
B_{R,T}:=\{u\in X_T:\|u\|_{X_T}\leq R\}\ \ \mathrm{and}\ \
d(u,v):=\|u-v\|_{Y_T}.
\end{equation*}
\end{lemma}

Now, we are ready to prove the following local existence result.

\begin{proposition}\label{pro5.1}
There exists $\gamma_0>0$ such that if $\varphi_0\in
H^1(\mathbb{R}^N)$ and $T\in(0,1]$ satisfy
\begin{equation}\label{e5.5}
\|e^{it\Delta}\varphi_0\|_{X_T}\leq \gamma_0,
\end{equation}
then there exists a unique integral solution $\varphi(t,x)$ to
(\ref{e5.1}) on the time interval $[0,T]$. Moreover $\varphi(t,x)\in
X_{p,r,T}$ for every $(p, r)\in S$ and satisfies the following
conservation laws:
\begin{equation}\label{e5.6}
E(\varphi(t))=E(\varphi_0),\ \|\varphi(t)\|_2=\|\varphi_0\|_2,\ \
\mathrm{for\ all}\ t\in[0,T].
\end{equation}
\end{proposition}

\begin{proof}
By modifying the proof of Proposition 3.3 in
\cite{JEANJEAN-JENDREJ}, we can show that there exists a unique
integral solution $\varphi(t,x)$ to (\ref{e5.1}) on the time
interval $[0,T]$ and $\varphi(t,x)\in X_{p,r,T}$ for every $(p,
r)\in S$. The proofs of the conservation laws (\ref{e5.6}) follow
the proofs of Propositions 1 and  2 in \cite{Ozawa 2006}, that can
be repeated mutatis mutandis in the context of (\ref{e5.1}).
\end{proof}

\subsection{Orbital stability}

\smallskip

Now we prove  Theorem \ref{thm6.2}.

\textbf{Proof of Theorem \ref{thm6.2}}. Since in the context of
(\ref{e5.1}), we have the local existence result (Proposition
\ref{pro5.1}),  the proof  of Theorem \ref{thm6.2} can be done by
repeating word by word Section 4 in \cite{JEANJEAN-JENDREJ} and we
omit it.

\subsection{Orbital instability}

In this subsection, we prove Theorem \ref{thm1.4}. For this aim, we
first give the following result.

\begin{lemma}\label{lem blow up}
Assume $N\geq 3$, $\alpha\in (0,N)$, $p=\bar{p}$, $q\in
(2,2+\frac{4}{N})$, $\alpha\geq N-4$ (i.e.,$\bar{p}\geq 2$) and
$\alpha< N-2$. Let $u\in S_a$ be such that
$E(u)<\inf_{\mathcal{P}_{a,-}}E(v)$ and let $\tau_u^-$ be the unique
global maximum point of $\Psi_u(\tau)$ determined in Lemma
\ref{lem6.6}. If $\tau_u^-<1$ and $|x|u\in L^2(\mathbb{R}^N)$, then
the solution $\psi(t,x)$ of (\ref{e1.2}) with initial value $u$
blows up in finite time.
\end{lemma}

\begin{proof}
We claim that
\begin{equation}\label{e8.15}
\mathrm{if}\ u\in S_a\ \mathrm{and}\ \tau_u^-\in (0,1),\
\mathrm{then}\ P(u)\leq E(u)-\inf_{\mathcal{P}_{a,-}}E(v).
\end{equation}
Indeed, by using the equality
\begin{equation*}
\Psi_u(\tau_u^-)=\Psi_u(1)+\Psi_u'(1)(\tau_u^--1)+\Psi_u''(\xi)(\tau_u^--1)^2,\
\ \mathrm{for\ some}\ \xi\in (\tau_u^-,1),
\end{equation*}
and noting that $P(u)<P(u_{\tau_u^-})=0$, $\Psi_u''(\xi)<0$ for
$\xi>\tau_u^-$, $\Psi_u'(1)=P(u)$ and $\Psi_u(1)=E(u)$, we obtain
that
\begin{equation*}
\inf_{\mathcal{P}_{a,-}}E(v)\leq \Psi_u(\tau_u^-)\leq E(u)-P(u),
\end{equation*}
which implies that (\ref{e8.15}) holds.

Now, let us consider the solution $\psi(t,x)$ with initial value
$u$. By Proposition \ref{pro5.1}, $\psi(t,x)\in C([0,T_{max}),
H^1(\mathbb{R}^N))$, where $T_{max}\in (0, +\infty]$ is the maximal
lifespan of $\psi(t,x)$. Since by assumption $\tau_u^-<1$, and the
map $u\mapsto \tau_u^-$ is continuous, we deduce that
$\tau_{\psi(t)}^-<1$ as well for $t$ small, say $t\in [0, t_1)$. By
(\ref{e8.15}), the assumption $E(u)<\inf_{\mathcal{P}_{a,-}}E(v)$,
and the conservation laws of mass and energy, we obtain that for
$t\in [0, t_1)$,
\begin{equation*}
P(\psi(t))\leq
E(\psi(t))-\inf_{\mathcal{P}_{a,-}}E(v)=E(u)-\inf_{\mathcal{P}_{a,-}}E(v)<-\delta.
\end{equation*}
Hence, $P(\psi(t_1))\leq -\delta$ and  then $\tau_{\psi(t_1)}^-<1$.
Hence, by continuity, the above argument yields
\begin{equation*}
P(\psi(t))\leq-\delta,\ \mathrm{for\ any}\ t\in
[0,T_{\mathrm{max}}).
\end{equation*}
To obtain a contradiction we recall that, since $|x|u\in
L^2(\mathbb{R}^N)$ by assumption, by the virial identity (see
Proposition 6.5.1 in \cite{Cazenave 2003}), the function
\begin{equation*}
\Phi(t):=\int_{\mathbb{R}^N}|x|^2|\psi(t,x)|^2dx
\end{equation*}
is of class $C^2$, with $\Phi''(t)=8 P(\psi(t))\leq -8\delta$ for
every $t\in [0,T_{\mathrm{max}})$. Therefore
\begin{equation*}
0\leq \Phi(t)\leq \Phi(0)+\Phi'(0)t-4\delta t^2\ \ \mathrm{for\
every}\ t\in [0,T_{\mathrm{max}}).
\end{equation*}
Since the right hand side becomes negative for $t$ large, this
yields  an upper bound on $T_{\mathrm{max}}$, which in turn implies
finite time blow up.
\end{proof}

\textbf{Proof of Theorem \ref{thm1.4}}. By Lemmas \ref{lem3.8} and
\ref{lem3.3}, $u$ satisfies (\ref{e1.5}) with some $\lambda<0$.
Next, we prove the strong instability of $e^{-i\lambda t}u(x)$. For
$s>1$, let $u_s:=s^{N/2}u(sx)$ and $\psi_s(t,x)$ be the solution to
(\ref{e1.2}) with initial value $u_s$. We have $u_s\to u$ strongly
in $H^1(\mathbb{R}^N)$ as $s\to 1^+$, and hence it is sufficiently
to prove that $\psi_s$ blows up in finite time. Let $\tau_{u_s}^-$
be defined by Lemma \ref{lem6.6}. Clearly $\tau_{u_s}^-=s^{-1}<1$,
and by the definition of $\tau_{u_s}^-$,
\begin{equation*}
E(u_s)<E((u_s)_{\tau_{u_s}^-})=E(u)=\inf_{\mathcal{P}_{a,-}}E(v).
\end{equation*}
By Theorem \ref{thm1.3} (3), $|x|u_s\in L^2(\mathbb{R}^N)$. Hence,
by Lemma \ref{lem blow up}, $\psi_s$ blows up in finite time. The
proof is complete.

\bigskip

\textbf{Acknowledgements.} This work is supported by the National
Natural Science Foundation of China (No. 12001403).



\end{document}